\documentclass{amsart}[12pt]
\usepackage{url,amsmath,amssymb,amsthm,amscd,tikz,graphicx,amssymb,epstopdf,amsfonts,mathtools}
\usepackage[top = 1.5in, bottom = 1.5in, left = 1in, right = 1in]{geometry}
\usepackage[toc,page]{appendix}
\newtheorem{case}{Case}
\newtheorem{theorem}{Theorem}[section]
\newtheorem*{main3}{Theorem 1.4}
\newtheorem{lemma}[theorem]{Lemma}
\newtheorem{corollary}[theorem]{Corollary}
\newtheorem{proposition}[theorem]{Proposition}
\newtheorem{definition}[theorem]{Definition}
\newtheorem*{ack}{Acknowledgements}
\newcommand{\Z}{\mathbb{Z}}

\newcommand{\F}{\mathbb{F}} 
\newcommand{\Q}{\mathbb{Q}}
\newcommand{\C}{\mathbb{C}}
\DeclareMathOperator{\pr}{pr}
\DeclareMathOperator{\Gal}{Gal}
\DeclareMathOperator{\AGL}{AGL}
\DeclareMathOperator{\GL}{GL}

\DeclareMathOperator{\im}{im}
\DeclareMathOperator{\tr}{tr}
\DeclareMathOperator{\Aut}{Aut}
\newcommand{\agl}[1]{\ensuremath{AGL_2(\mathbb{Z}/2^{#1}\mathbb{Z})}}
\newcommand{\gl}[1]{\ensuremath{GL_2(\mathbb{Z}/2^{#1}\mathbb{Z})}}
\newcommand{\ztwok}[1]{\ensuremath{(\mathbb{Z}/2^{#1}\mathbb{Z})}}
\newcommand{\ih}{Inductive Hypothesis: }
\newcommand{\artin}[2]{\genfrac{[}{]}{}{}{#1}{#2}}
\begin{document}

\title[The Density of Primes]{The Density of Primes Dividing a Particular Non-Linear Recurrence Sequence}
\author[Block Gorman]{Alexi Block Gorman}
\author[Genao]{Tyler Genao}
\author[Hwang]{Heesu Hwang}
\author[Kantor]{Noam Kantor}
\author[Parsons]{Sarah Parsons}
\author[Rouse]{Jeremy Rouse}
\thanks{All authors were supported by the NSF grant DMS-1461189.}

\address{Department of Mathematics, Wellesley College, Wellesley, MA 02481}
\email{ablockgo@wellesley.edu}
\address{Department of Mathematical Sciences, Florida Atlantic University, Boca Raton, FL 33431}
\email{tgenao2013@fau.edu}
\address{Department of Mathematics, Princeton University, Princeton, NJ 08544}
\email{hshwang@princeton.edu}
\address{Department of Mathematics, Emory University, Atlanta, GA 30322}
\email{noam.kantor@emory.edu}
\address{Department of Mathematics, Wake Forest University, Winston-Salem, NC 27109}
\email{parssy12@wfu.edu}
\email{rouseja@wfu.edu}

\begin{abstract}
Define the ECHO sequence $\{b_n\}$ recursively by $(b_0,b_1,b_2,b_3)=(1,1,2,1)$ and for $n\geq 4$,
$$b_n=\begin{cases}
\dfrac{b_{n-1}b_{n-3}-b_{n-2}^2}{b_{n-4}} &\mathrm{if}~ n\not\equiv 0\pmod 3,\\
\dfrac{b_{n-1}b_{n-3}-3b_{n-2}^2}{b_{n-4}} &\mathrm{if}~ n\equiv 0\pmod 3.
\end{cases}$$\\
We relate this sequence $\{b_n\}$ to the coordinates of points on the elliptic curve $E:y^2+y=x^3-3x+4$. We use Galois representations attached to $E$ to prove that the density of primes dividing a term in this sequence is equal to $\frac{179}{336}$. Furthermore, we describe an infinite family of elliptic curves whose Galois images match those of $E$.
\end{abstract}

\maketitle

\section{Introduction and Statement of Results}
Given an integer sequence, it is natural to ask which primes divide at least one term in the sequence. More generally, what is the density of such primes? A theorem of Hasse shows that the density of primes dividing numbers of form $2^n+1$ is $\frac{17}{24}$, a result that is intimately connected to a question asked by Sierpinski \cite{MR0103854} on the multiplicative order of 2 in finite fields. Lagarias \cite{MR789184} (see the correction \cite{MR1251907}) extended Hasse\textsc{\char13}s methods to show that the density of primes dividing a term of the Lucas sequence is $\frac{2}{3}$. Hasse and Lagarias\textsc{\char13} methods include taking field extensions created from the characteristic polynomial of sequences and analyzing the resulting Galois groups, which depend entirely on the behavior of the sequence in $\F_p$. The Chebotarev Density Theorem can then be applied to these groups to calculate numerical values.

In \cite{MR2640290}, Jones and Rouse extend this theme and employ the techniques of Galois theory to study sequences attached to elliptic curves. It has long been known that Galois representations attached to an elliptic curve encode much of the arithmetic of points on the curve. Classically, the Galois representation is constructed using the action of the absolute Galois group on the $N$-torsion points of the identity on the curve. In \cite{MR2640290}, the authors modify the classical representation to study $N$-division points of a non-identity point $P$. Fix a prime $\ell$;  when the Galois group associated to the $\ell$-division points is as large as possible, they use this arboreal representation to calculate the density of primes $p$ for which the reduction of $P$ modulo $p$ has order coprime to $\ell$.

In the present paper, we extend the bridge between sequences, Galois theory, and elliptic curves by using Galois representations to determine the density of primes dividing the following non-linear, integral recurrence sequence. The motivation for our work was a sequence given to us by our mentor with properties reminiscent of Somos-like sequences that, as we shall see, diverged from their collection of elliptic curves. We call this novel sequence the ECHO sequence, which is defined as follows.

\begin{definition}
We define the ECHO sequence $\{b_n\}$ recursively by $(b_0,b_1,b_2,b_3)=(1,1,2,1)$ and for $n\geq 4$,
$$b_n=\begin{cases}
\dfrac{b_{n-1}b_{n-3}-b_{n-2}^2}{b_{n-4}} &\mathrm{if}~ n\not\equiv 0\pmod 3,\\
\dfrac{b_{n-1}b_{n-3}-3b_{n-2}^2}{b_{n-4}} &\mathrm{if}~ n\equiv 0\pmod 3.
\end{cases}$$
\end{definition}

Consider the elliptic curve $E : y^{2} + y = x^{3} - 3x +4$ and $P = (4,7) \in E(\Q)$. We show in Lemma~\ref{twoNplusthree} that
$$(2n+1)P=\left(\frac{g(n)}{b_n^2},\frac{f(n)}{b_n^3}\right),$$
where $g(n)=2b_n^2-b_{n-3}b_{n+3}$, and
$$f(n)=\begin{cases}
b_n^3+3b_{n-1}^2b_{n+2}&\mathrm{if}~n\equiv 0\pmod 3,\\
b_n^3+b_{n-1}^2b_{n+2}&\mathrm{if}~n\equiv 1\pmod 3,\\
b_n^3+9b_{n-1}^2b_{n+2}&\mathrm{if}~n\equiv 2\pmod 3.
\end{cases}$$

This equation relates the question of whether a prime $p$ divides some element of the sequence $\{b_n\}$ to whether $P$ in $E(\mathbb{F}_p)$ has odd order, and Galois theory can be readily applied to the latter problem. We explain this connection in Section 3.

The motivation for looking at this specific curve and sequence is that the Galois representations of the $2^k$ division fields are not as large as possible, which was the case in the work of Jones and Rouse. Their work shows that a sequence arising from an elliptic curve with surjective 2-adic representation has divisibility-density $\frac{11}{21}$.

In comparison, below are some approximate densities for our sequence. We define $\pi'(x)$ as the number of primes $p$ less than $x$ such that $p$ divides some term of the sequence.

\begin{center}\begin{tabular}{c|ccc}
$x$ & $\pi'(x)$ & $\pi(x)$ & $\frac{\pi'(x)}{\pi(x)}$\\
\hline
$10$ & $3$ & $4$ & $0.75$\\
$10^{2}$ & $13$ & $25$ & $0.52$\\
$10^{3}$ & $91$ & $168$ & $0.541666667$\\
$10^{4}$ & $636$ & $1229$ & $0.517493897$\\
$10^{5}$ & $5118$ & $9592$ & $0.533569641$\\
$10^{6}$ & $41856$ & $78498$ & $0.533211037$\\
$10^{7}$ & $354158$ & $664579$ & $0.532905794$\\
$10^{8}$ & $3069170$ & $5761455$ & $0.532707450$\\
$10^{9}$ & $27092923$ & $50847534$ & $0.532826685$\\
$10^{10}$ & $242426819$ & $455052511$ & $0.532744712$\\
$10^{11}$ & $2193850226$ & $4118054813$ & $0.532739443$
\end{tabular}\end{center}
This data suggests that the fraction converges to a limit different from $\frac{11}{21}= 0.\overline{523809}$. In this paper, we prove that this limit exists, and we compute its value.
\begin{theorem}\label{main}
\[\lim_{x \to \infty} \frac{\pi'(x)}{\pi(x)} = \frac{179}{336} = 0.5327\overline{380952}.\]
\end{theorem}

In the process, we obtain a general result on the image of the arboreal representation:

\begin{theorem}\label{main2}
Let $E$ be an elliptic curve over $\Q$ and let $P$ be a rational point on $E$. Suppose further that the classical Galois representation is surjective, and that $P$ has no rational 2-division points. Then there are only two possibilities for the image of the $2$-adic arboreal representation, up to conjugacy: the entirety of $\agl{k}$, and an index four subgroup that we denote by $H_k$. As a consequence, the density of primes $p$ for which the reduction of $P$ modulo $p$ has odd order is either $\frac{11}{21}$ or $\frac{179}{336}$.
\end{theorem}

Last, we construct a surface whose rational points correspond to pairs $(E/\Q,P \in E(\Q))$ such that the 2-adic arboreal representation attached to $(E,P)$ has image contained in the subgroup $H_k$. We use this surface to prove the following:

\begin{main3}
There exists a one-parameter family of curve-point pairs $(E,P)$ such that infinitely many pairs in the family have arboreal image conjugate to $H_k$. In particular, infinitely many pairs $(E,P)$ in our family have the property that the density of primes $p$ for which the reduction of $P$ modulo $p$ has odd order is $\frac{179}{336}$.
\end{main3}

\begin{ack}
We would like to acknowledge Sagemath \cite{sage} and Magma \cite{MR1484478}, which we extensively used for computations. We would also like to thank Wake Forest University for providing resources for this research project. Thanks also to Michael Somos for providing input on the structure of the ECHO sequence. We also
wish to thank the anonymous referee for helpful comments.
\end{ack}

\section{Background}
An elliptic curve $E$ is a non-singular cubic curve defined over a field $K$ with a $K$-rational
point on it. Such a curve has an equation of the form $y^{2} + a_{1} xy + a_{3} y = x^{3} + a_{2} x^{2} + a_{4} x + a_{6}$. Let $E(K)$ be the set of points on the projective closure of
$E$ in $\mathbb{P}^{1}(K)$. This set has the structure of an abelian group: if $P, Q \in E(K)$,
and $R = (x,y)$ is the third intersection of the line through $P$ and $Q$ with $E$, then $P+Q = (x,-y-a_{1}x - a_{3})$. We denote the abelian group structure on $E(K)$ additively.

Given an elliptic curve $E/\Q$, we say that $E$ has good reduction at a prime $p$ if
$E/\F_{p}$ is non-singular. Otherwise, we say that $E$ has bad reduction at $p$.

If $P \in E(\C)$, the $m$-division points of $P$ are defined to be the points $Q$ for which $mQ = P$. We let $[m^{-1}]P$ denote the set of such points; this set has $m^{2}$ elements.

If $K/\Q$ is a finite extension, let $\mathcal{O}_{K}$ denote the ring of algebraic integers in $K$. A prime number $p$ ramifies in $K$ if $p \mathcal{O}_{K}$ factors as $\prod_{i=1}^{r} \mathfrak{p}_{i}^{r_{i}}$, where the $\mathfrak{p}_{i}$ are prime ideals of $\mathcal{O}_{K}$ and $r_{i} > 1$ for some $i$. If, in addition, $K/\Q$ is Galois and $p$ is unramified in $K/\Q$, then we define
$\artin{K/\Q}{\mathfrak{p}_{i}}$ to be the unique element $\sigma \in \Gal(K/\Q)$ so that
\[
\sigma(\alpha) \equiv \alpha^{p} \pmod{\mathfrak{p}_{i}}
\]
for all $\alpha \in \mathcal{O}_{K}$. The set $\left\{ \artin{K/\Q}{\mathfrak{p}_{i}} \right\}$ is
a conjugacy class in $\Gal(K/\Q)$ which we denote by $\artin{K/\Q}{p}$. We will use the Chebotarev density theorem to prove Theorem~\ref{main}. This theorem states the following.
\begin{theorem}[\cite{MR3236783}, edition 2, Theorem 8.17, page 153 or \cite{MR2061214}, page 143]
If $C \subseteq \Gal(K/\Q)$ is a conjugacy class, then
\[
\lim_{x \to \infty} \frac{\left|\{ p~\text{prime} : p \leq x, \artin{K/\Q}{p} = C \}\right|}{\pi(x)} = \frac{|C|}{|\Gal(K/\Q)|}.
\]
\end{theorem}

If $E/\Q$ is an elliptic curve, let $\Q(E[m])$ denote the field obtained by adjoining to $\Q$
all the $x$ and $y$-coordinates of points in $E[m]$, the set of points on $E$ of order dividing $m$. We have $E[m](\C) \cong (\Z/m\Z)^{2}$. Because the group law on $E$ is given by rational functions, it commutes with the action of the Galois group, and as a consequence we have the classical Galois representation $\rho_{E,m} : \Gal(\Q(E[m])/\Q) \to \Aut(E[m]) \cong \GL_{2}(\Z/m\Z)$. This Galois representation has the properties that if $\sigma$ is the Artin symbol of a prime ideal above $p$, then
\[
  \tr~\rho_{E,m}(\sigma) \equiv p+1 - \# E(\F_{p}) \pmod{m}
\]
and $\det \rho_{E,m}(\sigma) \equiv p \pmod{m}$. Moreover, by the Neron-Ogg-Shafarevich criterion \cite[Theorem VII.7.1, page 184]{MR2514094}, $\Q(E[m])/\Q$ is ramified only at primes that divide $m$ and primes for which $E/\F_{p}$ has bad reduction.

If $E/\Q$ is an elliptic curve and $P \in E(\Q)$, the extension $\Q([m^{-1}] P)/\Q$ is Galois, and if $\beta$ is a chosen element of $[m^{-1}] P$, we have an arboreal Galois representation $\omega_{E,m} : \Gal(\Q([m^{-1}] P)/\Q) \to E[m] \rtimes \Aut(E[m])$ given by $\omega_{E,m}(\sigma) = (\sigma(\beta) - \beta, \sigma|_{E[m]})$. Proposition VIII.1.5 on page 193 \cite{MR2514094} implies that $\Q([m^{-1}] P)/\Q$ is ramified
only at primes at which $E$ has bad reduction and primes dividing $m$. We are most interested
in the case that $m = 2^{k}$, $E : y^{2} + y = x^{3} - 3x + 4$, and $P = (4,7)$. This elliptic curve has bad reduction only at $p = 3$ and $p = 5$. We let $K_{k} = \Q([2^{-k} ]P)$ and define
\[
  \omega_{k} : \Gal(K_{k}/\Q) \to \agl{k} := (\Z/2^{k} \Z)^2 \rtimes \gl{k}.
\] 
We represent elements of $\agl{k}$ as pairs $(\vec{v}, M)$ of column vectors $\vec{v}$ and $2 \times 2$ matrices $M$. The group law is given by
\[
  (\vec{v}_{1}, M_{1}) * (\vec{v}_{2}, M_{2}) = (\vec{v}_{1} + M_{1} \vec{v}_{2}, M_{1} M_{2}).
\]

If $E/K$ is an elliptic curve and $K$ is a number field, then for each prime ideal $\mathfrak{p}$ above a prime $p$ of $\mathcal{O}_{K}$ for which $E$ has good reduction, there is a reduction homomorphism $\phi_{\mathfrak{p}} : E(K) \to E(\mathcal{O}_{K}/\mathfrak{p})$. A point $Q$ in $E(K)$ maps to the point at infinity in $E(\mathcal{O}_{K}/\mathfrak{p})$ if and only if the $x$ and $y$-coordinates of $Q$ have negative $\mathfrak{p}$-adic valuation. Another useful
property of these reduction maps is that \cite[Proposition VII.3.1, page 176]{MR2514094} the only points of finite order in $E(K)$ that can be contained in $\ker \phi_{\mathfrak{p}}$ have order a power of $p$.

\section{Connection between the ECHO sequence, $E$, and $P$}\label{sequences}

We first translate the problem of finding the density of primes dividing a term in the sequence into a question about elliptic curves by showing that $p|b_n$ for some $n\geq 0$ if and only if $P=(4,7)\in E(\mathbb{F}_p)$ has odd order.

To prove this, we introduce several lemmas that establish integrality and coprimality conditions on our sequence, then prove the formula established in Lemma~\ref{twoNplusthree}. As a quick remark, one can extend the sequence into negative indices by noting that
$$b_n= \begin{cases}
\frac{b_{n+3}b_{n+1}-b_{n+2}^2}{b_{n+4}}&\textrm{if } n\not \equiv 2\pmod 3,\\
\frac{b_{n+3}b_{n+1}-3b_{n+2}^2}{b_{n+4}}&\textrm{if } n\equiv 2\pmod 3.
\end{cases}$$
In fact, induction shows that the ECHO sequence is symmetric by the relation $b_n=-b_{-(n+1)}$. Thus, the formulas proven here can extend in both directions, and we may use negative indices to establish base cases as needed.

\begin{lemma}\label{periodic}
The denominator of $b_n$ is coprime to 3 and $\lbrace b_n\pmod 3\rbrace$ is periodic.
\end{lemma}
\begin{proof}
Write $b_n=\frac{a_n}{c_n}$ where $a_n,c_n\in \mathbb{Z}$ are coprime. Interpret $b_n$ as $b_n=a_nc_n^{-1}\in\mathbb{F}_3$. We claim that for $n\equiv 0\pmod 9$,
$$(b_{n},b_{n+1},b_{n+2}, b_{n+3}, b_{n+4}, b_{n+5}, b_{n+6}, b_{n+7}, b_{n+8})\equiv (1,1,2,1,0,2,1,2,2)\pmod 3.$$

We will use induction to prove this; the base cases are simple to check. Assume that the sequence is periodic for $k<n$, and suppose that $n\equiv 0\pmod 9$. Then the first eight congruences are easily checked.

To calculate $b_{n+8}$, let $a=b_n, b=b_{n+1},c=b_{n+2},d=b_{n+3}$, and $e=b_{n+4}$, and write $b_{n+5}$ through $b_{n+7}$ in terms of $b,c,d,e$. Then writing $b_{n+8}$ in terms of $b,c,d,e$, we have that $$b_{n+8}=\frac{Ke-bd^7+c^2d^6}{b^3c^2de},$$ where $Ke=-9b^3de^4 - 3b^2c^2e^4 + 9b^2cd^2e^3 - 6b^2d^4e^2 + bc^3de^3 - 3bc^2d^3e^2 + 3bcd^5e - c^5e^3 + 3c^4d^2e^2  - 3c^2d^4e$. Then
\begin{align*}
  b_{n+8} &= \frac{Ke-bd^7+c^2d^6}{b^3c^2de}= \frac{Ke-d^6(bd-c^2)}{b^3c^2de}= \frac{K-ad^6}{b^3c^2d}\\&\equiv K-a\equiv K+2\equiv 2\pmod 3,
\end{align*}
hence $b_{n+8}\equiv 2\pmod 3$.
\end{proof}

\begin{lemma}\label{integralSeq}
For $n\geq 0$ we have that $b_n\in \mathbb{Z}$ and $\gcd(b_n,b_{n-3})=\gcd(b_n,b_{n-2})=\gcd(b_n,b_{n-1})=1$.
\end{lemma}
\begin{proof}
We will prove this by induction on the index $n$:

The base cases can be easily checked. Assume for $k<n$ that $b_k\in \mathbb{Z}$ and that $b_k$ is coprime to $b_{k-3},b_{k-2}$, and $b_{k-1}$. Let $b_{n-7}=a,b_{n-6}=b,b_{n-5}=c$. If $n-4\equiv 1\pmod 3:$
$$\begin{array}{cccccccccc}
b_{n-4}b_{n-8}&=&b_{n-5}b_{n-7}&-&b_{n-6}^2&\Rightarrow& ac-b^2&\equiv& 0&\pmod{ b_{n-4}},
\\b_{n-3}b_{n-7}&=&b_{n-4}b_{n-6}&-&b_{n-5}^2&\Rightarrow& b_{n-3}&\equiv& \frac{-c^2}{a}&\pmod{ b_{n-4}},
\\b_{n-2}b_{n-6}&=&b_{n-3}b_{n-5}&-&3b_{n-4}^2&\Rightarrow& b_{n-2}&\equiv& \frac{-c^3}{ab}&\pmod{b_{n-4}},
\\b_{n-1}b_{n-5}&=&b_{n-2}b_{n-4}&-&b_{n-3}^2&\Rightarrow& b_{n-1}&\equiv& \frac{-c^3}{a^2}&\pmod{b_{n-4}}.
\end{array}$$

Note that
$$b_{n-1}b_{n-3}-b_{n-2}^2\equiv \frac{-c^3}{a^2}\cdot \frac{-c^2}{a}-\frac{c^6}{a^2b^2}\equiv\frac{c^5b^2-c^6a}{a^3b^2}\equiv \frac{-c^5(ac-b^2)}{a^3b^2}\pmod {b_{n-4}},$$
hence $b_{n-1}b_{n-3}-b_{n-2}^2$ is congruent to 0 mod $b_{n-4}$. Therefore, $b_{n-4}|(b_{n-1}b_{n-3}-b_{n-2}^2)$, and so $b_n\in\mathbb{Z}$.

The other two cases can be similarly checked.

A straightforward argument can be made using Lemma ~\ref{periodic} to show that if $p|b_n$, then $p\nmid b_{n-1},p\nmid b_{n-2}$, and $p\nmid b_{n-3}$.
\end{proof}

We now turn to connecting the sequence to our elliptic curve. One preliminary is required.

\begin{lemma}\label{h}
For $n\geq 0$, define
$$h(n)=\begin{cases}
b_{n-3}^2b_{n}^2 + b_{n-3}b_{n-1}^3 + 3b_{n-2}^3b_{n} - 3b_{n-2}^2b_{n-1}^2 & \mathrm{if}~n\equiv 0\pmod 3,\\
3b_{n-3}^2b_{n}^2 + b_{n-3}b_{n-1}^3 + b_{n-2}^3b_{n} - b_{n-2}^2b_{n-1}^2 &  \mathrm{if}~n\equiv 1\pmod 3,\\
b_{n-3}^2b_{n}^2 + 3b_{n-3}b_{n-1}^3 + b_{n-2}^3b_{n} - 3b_{n-2}^2b_{n-1}^2 & \mathrm{if}~n\equiv 2\pmod 3.
\end{cases}$$
Then $h(n)=0$ for all $n\in\mathbb{N}$.
\end{lemma}
\begin{proof} We proceed by induction. The base case of $n=1$ is easy to verify by computation.

Now suppose that the claim is true for $k<n$. If $n\equiv 1\pmod 3$, then
\begin{align*}
h(n)&=3b_{n-3}^2b_{n}^2 + b_{n-3}b_{n-1}^3 + b_{n-2}^3b_{n} - b_{n-2}^2b_{n-1}^2\\
&=b_n(b_{n-2}^3+3b_{n-3}^2b_n)+b_{n-1}^2(b_{n-3}b_{n-1}-b_{n-2}^2)\\
&=b_n\left(b_{n-2}^3+3b_{n-3}^2\cdot\dfrac{b_{n-3}b_{n-1}-b_{n-2}^2}{b_{n-4}}\right)+b_{n-1}^2b_{n-4}\cdot \dfrac{(b_{n-3}b_{n-1}-b_{n-2}^2)}{b_{n-4}}\\
&=b_n\left(b_{n-2}^3+3b_{n-3}^2\cdot\dfrac{b_{n-3}b_{n-1}-b_{n-2}^2}{b_{n-4}}\right)+b_{n-1}^2b_{n-4}b_n\\
&=\dfrac{b_n}{b_{n-4}}\cdot (b_{n-4}^2b_{n-1}^2 + b_{n-4}b_{n-2}^3 + 3b_{n-3}^3b_{n-1}-3b_{n-3}^2b_{n-2}^2)\\
&=\dfrac{b_n}{b_{n-4}}\cdot h(n-1).\end{align*}

The equations for cases $n\equiv 1,2\pmod 3$ similarly reduce, and in general,
$h(n)=\dfrac{b_n}{b_{n-4}}\cdot h(n-1).$ Therefore, by induction, $h(n)=0$ for all $n\geq 0$.
\end{proof}

This suffices to prove the following.
\begin{lemma}\label{twoNplusthree}
Define $P = (4,7)$ on E : $y^{2} + y = x^{3} - 3x + 4$. Then for $n\geq 0$,
$$(2n+1)P=\left(\frac{g(n)}{b_n^2},\frac{f(n)}{b_n^3}\right),$$
where $g(n)=2b_n^2-b_{n-3}b_{n+3}$, and
$$f(n)=\begin{cases}
b_n^3+3b_{n-1}^2b_{n+2}&\mathrm{if}~n\equiv 0\pmod 3,\\
b_n^3+b_{n-1}^2b_{n+2}&\mathrm{if}~n\equiv 1\pmod 3,\\
b_n^3+9b_{n-1}^2b_{n+2}&\mathrm{if}~n\equiv 2\pmod 3.
\end{cases}$$
Furthermore, each coordinate in this expression is in reduced form. \end{lemma}

As a consequence of Lemma ~\ref{twoNplusthree}, the proof of the following is easily obtained.

\begin{corollary}\label{seqAndOddOrder}
Suppose $p$ is a prime for which $E(\mathbb{F}_p)$ has good reduction. Then $p|b_n$ for some $n\geq 0$ if and only if $P\in E(\mathbb{F}_p)$ has odd order.
\end{corollary}

The proof of Lemma~\ref{twoNplusthree} is as follows.

\begin{proof}
Proceed by induction. For $n=0,1,2$, $(2\cdot 0+1)P=(4,7)$, $(2\cdot 1 + 1)P=(-1,2)$, and $(2\cdot 2+1)P=(\frac{1}{4},\frac{-19}{8})$; thus the base cases are true.

Induction Hypothesis: Suppose that the claim is true for all $n\leq k\in\mathbb{N}$. Thus, 
$$(2k+1)P=\left(\dfrac{2b_{k}^2-b_{k-3}b_{k+3}}{b_{k}^2},\dfrac{f(k)}{b_{k}^3}\right).$$
This, along with $2P=(1,1)$, implies that
$$(2(k+1)+1)P=(2k+1)P+2P=\left(\dfrac{2b_{k}^2-b_{k-3}b_{k+3}}{b_{k}^2},\dfrac{f(k)}{b_{k}^3}\right)+(1,1).$$
By substituting the equation of the line in for $y$ and examining the $x^2$ coefficient, the $x$-coordinate of $(2(k+1)+1)P$ is $\left(\dfrac{\beta-1}{\alpha-1}\right)^2-1-\alpha$ for $\alpha=\dfrac{2b_{k}^2-b_{k-3}b_{k+3}}{b_{k}^2}$ and $\beta=\dfrac{f(k)}{b_{k}^3}$. Let $\alpha'=\dfrac{2b_{k+1}^2-b_{k-2}b_{k+4}}{b_{k+1}^2}$; we hope to show that $\alpha'$ is the $x$-coordinate of $(2(k+1)+1)P$. In doing this, it suffices to show that
\begin{equation}\label{one}\left(\dfrac{\beta-1}{\alpha-1}\right)^2-1-\alpha-\alpha'=0.\end{equation}
Let the expression on the left be called expression (\ref{one}).

Using our recursive definition for $\{b_n\}$, we may substitute higher valued sequence points in terms of lower points in (\ref{one}) so that (\ref{one}) is written as a rational fraction in terms of only the four sequence points $b_{k-3}$ through $b_{k}$. Magma explicitly tells us that $h(k)=0$ is a factor of the numerator (for the same $h(n)$ defined in Lemma ~\ref{h}). Therefore, $(\ref{one})=0$, and so the $x$-coordinate is $\alpha'=\dfrac{2b_{k+1}^2-b_{k-2}b_{k+4}}{b_{k+1}^2}$.

Next, because $E:y^2+y=x^3-3x+4$, the addition of the two points is the third intersection point reflected over the line $y=-\frac{1}{2}$. This action sends any point $(\gamma,\delta)$ to $(\gamma,-\delta-1)$.

The third point is given by substitution of the newly found $x=\alpha'$. Thus, using the line equation it suffices to show that $$-\dfrac{\beta-1}{\alpha-1}(\alpha'-1)-2=\dfrac{f(k+1)}{b_{k+1}^3}$$
\begin{equation}\label{two}
\Leftrightarrow \dfrac{\dfrac{f(k)}{b_{k}^3}-1}{\dfrac{2b_{k}^2-\cdot b_{k-3}b_{k+3}}{b_{k}^2}-1}\cdot\left(\dfrac{2b_{k+1}^2-b_{k-2}b_{l+4}}{b_{k+1}^2}-1\right)+2+\dfrac{f(k+1)}{b_{k+1}^3}=0.
\end{equation}

Call the left hand side expression of the above (\ref{two}). In similar fashion to finding the $x$-coordinate, (\ref{two}) may be re-written with only terms $b_{k-3}$ through $b_{k}$ by substitution. Magma tells us this has $h(k)$ as a factor of the numerator. Thus, (\ref{two}) is always $0$ and $\dfrac{f(k+1)}{b_{k+1}^3}$ is the $y$-coordinate as desired, which completes the induction. Therefore, $(2n+1)P=\left(\dfrac{g(n)}{b_n^2},\dfrac{f(n)}{b_n^3}\right)$. 

To see that this is in reduced form, note that Lemma ~\ref{integer} implies that $\gcd(b_n,b_{n\pm3})=1$. Thus, the $x$-coordinate $\dfrac{2b_n^2-b_{n-3}b_{n+3}}{b_n^2}$ is already in reduced form, and hence the $y$-coordinate, having $b_n^3$ in its denominator, is also in reduced form.
\end{proof}

\section{Computing $\lim\limits_{x\rightarrow\infty}{\frac{\pi'(x)}{\pi(x)}}$ up to $10^{11}$} 
This connection between the sequence and elliptic curve gives us a way to explicitly compute $\pi'(x)$ for $x<\infty$. We test whether $P$ has odd order for those primes less than $x$ with good reduction.

If $p$ is a prime of good reduction, Proposition 3.5 allows us to determine whether or not $p$ divides a term in the sequence, depending on if $P \in E(\mathbb{F}_{p})$ has odd order. We must also consider the cases where the curve has bad reduction. Since our curve has discriminant $-3^5\cdot 5^2$, the only bad reductions are at $p=3,5$. As a result of Lemma ~\ref{periodic}, $3$ divides a term in our sequence, and so is included in our calculation of $\pi'(x)$. In contrast, we show below that $5$ does not divide any term in the sequence, and so is not included in the computation of  $\pi'(x)$. 

\begin{lemma}\label{periodic5} 
No term in the sequence is divisible by 5.
\end{lemma}
\begin{proof}
For the first 24 terms of the sequence, no term of the sequence is congruent to 0 modulo 5. By induction, it can be shown that no element of the sequence is divisible by 5 because the sequence is periodic modulo 5, with period 24. Thus, the desired result follows.
\end{proof}

In order to determine the primes less than $10^{11}$ for which $P \in E(\mathbb{F}_p)$ has odd order, we wrote code using PARI/GP \cite{PARI2} to perform computations on a server with 24 CPUs, each of which were an Intel Xeon E5-2630 2.3 GHz processor. By dividing the task into 24 different processors, we essentially divided $10^{11}$ by 24 and computed the number of primes for which $P \in E(\mathbb{F}_p)$ has odd order for each of the 24 ranges. The completion of all computations required approximately 4 days. Upon determining the primes of odd order, we then calculated means using the equation:
$$\text{mean} = \pi'(x)/\text{number of primes},$$ where $x=10^k$ for $1 \leq k \leq 11$ and $\pi'(x)$ equals the number of primes less than $x$ for which $P \in E(\mathbb{F}_p)$ has odd order.

With knowledge of the prime factorization of the denominator of the fraction above, we made calculations to hypothesize which fraction correctly described the number of all primes for which $P \in E(\mathbb{F}_p)$ has odd order. It was known, after studying the work of Rouse and Jones \cite{MR2640290}, that the prime factorization of the denominator should contain a power of 2 and some factor of 63. Therefore, there were approximately six cases, as the power of 2 varied, and 3 and or 7 were included in the factorization. With the theoretical values for each mean depicting a value close to 0.53273, the mean for the $10^{11}$ case was used to find the numerator of the fraction by multiplying the denominator by 0.532739443. Considering the mean as the probability of there being the computed number of primes for which $P \in E(\mathbb{F}_p)$ has odd order and $1-\text{mean}$ as the probability of there being the number of primes of even order, we computed the standard deviation, variance, and standard error of the mean to determine the certainty of the data, and we used the $z$-scores to better evaluate each fraction that we tested. After studying the $z$-scores for each of the six cases, we determined that the denominator was 336, with a prime factorization of $2^4\cdot 3\cdot 7$. We calculated the numerator to be 179, using the mean of the $10^{11}$ case, and so we predicted the fraction to be $\frac{179}{336}$, yielding a decimal of 0.532738095 with $z$-scores of an absolute value less than 0.7. The most positive $z$-score was 0.582 for the $10^4$ case and the most negative z-score was $-0.674$ for the $10^9$ case. We determined that the same $z$-scores and fraction were the best representation of three of the six different cases. Thus, our data supports the notion that the fraction $\frac{179}{336}$ is an adequate model for the density of primes for which $P \in E(\mathbb{F}_p)$ has odd order.


\section{Relating odd order to the arboreal representation.}

We now define a condition equivalent to $P=(4,7)$ having odd order in $E(\mathbb{F}_p)$ for a prime integer $p$. This will relate the current problem to arboreal Galois representations, which gives us more tools to work with. From now on, we consider only primes not equal to $2$, $3$, or $5$, which are the ramified primes or primes of bad reduction.

We fix the following terminology. Let $\{ \beta_{k(i)} \}$ be the set of elements of $E(\mathbb{C})$ for which $2^k \cdot \beta_{k(i)} = P$ for $i \in \{ 1,\ldots,4^k \}$. Let $K_k$ be field obtained by adjoining to $\mathbb{Q}$ all the $x$ and $y$ coordinates of such points $\{ \beta_{k(i)} \}$. Also, let $x(\beta_{k(i)})$ and $y(\beta_{k(i)})$ denote the $x$ and $y$-coordinates of $\beta_{k(i)}$.

\begin{theorem} \label{iff}
Let $p$ be a prime unramified in $K_k$, $(\vec{v},M) \in \omega_k\left(\left[ \frac{K_k/\mathbb{Q}}{p} \right]\right)$, and $\det(M-I) \not\equiv 0 \pmod{2^k}$. Then the point $P$ has odd order in $E(\mathbb{F}_p)$ if and only if $\vec{v}$ is in the column space of $(M-I)$.
\end{theorem}

\begin{proof}
($\Rightarrow$) Suppose that $P$ has odd order in $E(\mathbb{F}_p)$ for some prime $p$. Note that an element $a$ in a finite abelian group $G$ has odd order if and only if for any positive integer $k$ there exists an element $\beta_{k}\in G$ such that $2^k \cdot \beta_k = a$. Thus, for every $k$, there exists $\beta_{k(n)} \in E(\mathbb{F}_p)$ such that $2^k \cdot \beta_{k(n)} = P$ for some $n \in \{1, \ldots, 4^k \}$ for each integer $k$. There is a bijection between the $2^k$-division points of $P$ over $\F_p$ and the $2^k$-division points of $P$ over $K_k$, because the difference between any two $2^k$-division points of $P$ is a $2^k$-torsion point and the reduction modulo $p$ map is injective on torsion points with order coprime to the odd prime $p$. Therefore we know we may identify this $\F_p$-point $\beta_{k(n)}$ with one $\beta_{k(n)}$ in $K_k$.

Since the reduction of $\beta_{k(n)}$ is in $E(\mathbb{F}_p)$, it is fixed by the Frobenius automorphism, which corresponds to some element of $\left[ \frac{K_k/ \mathbb{Q}}{p} \right]$. Let $\sigma_p=\left[ \frac{K_k/ \mathbb{Q}}{\mathfrak{p}} \right]$ for some prime ideal $\mathfrak{p}$ over $p$ be such an element. Then $\sigma_p(\beta_{k(n)}) - \beta_{k(n)} \equiv 0 \pmod{\mathfrak{p}}$. But, $\sigma_p(\beta_{k(n)}) - \beta_{k(n)}$ is $2^k$-torsion and the reduction map is injective on torsion coprime to $p$, so in fact, $\sigma_p(\beta_{k(n)}) = \beta_{k(n)}$.

Recall the homomorphism $\omega_k : $ Gal$(K_k/ \mathbb{Q}) \rightarrow \agl{k}$. Define the action of $g=(\vec{v},M) \in \agl{k}$ on $\vec{x} \in (\mathbb{Z} / 2^k \mathbb{Z})^2$ by
$$ g(\vec{x}) = M \cdot \vec{x} + \vec{v}. $$



We fix some element $\beta_{k(m)}\in [2^{k}]^{-1}P$ and then $\omega_{E,2^k}(\sigma)=(\sigma(\beta_{k(m)})-\beta_{k(m)},\sigma|_{E[2^k]})$.
Note that the difference of any $2^k$-division points of $P$ is a $2^k$-torsion point. The set $\{\beta_{k(i)}-\beta_{k(m)}\}$ is thus equal to the set of $2^k$-torsion points, which is isomorphic to $(\mathbb{Z}/2^k\mathbb{Z})^2$. Since we have a bijection between elements $\beta_{k(i)}$ and $\beta_{k(i)}-\beta_{k(m)}$, we have a bijection between $\{\beta_{k(i)}\}$ and $\{\vec{x}\in (\mathbb{Z}/2^k\mathbb{Z})^2\}$. Furthermore, it is easy to check that this bijection respects the group actions involved.

We have $\sigma_p = \left[ \frac{K_k/ \mathbb{Q}}{\mathfrak{p}} \right]$ and $\beta_{k(n)} \in E(\mathbb{C})$ such that $\sigma_p(\beta_{k(n)}) =\beta_{k(n)}$. Let $\omega_{E,2^k}(\sigma_p)=(\vec{v},M)$ and $\vec{x}$ correspond to $\beta_{k(n)}$ by our bijection. Finally, $\sigma_p(\beta_{k(n)})=\beta_{k(n)}$, so $\omega_{E,2^k}(\sigma_p)(\vec{x})=\vec{x}$, and $M\cdot \vec{x}+\vec{v}=\vec{x}$. $(M-I)\vec{x}=-\vec{v}$. Equivalently, $\vec{v}$ is in the column space of $M-I$.\\

($\Leftarrow$) Conversely, suppose there exists an element $(\vec{v},M)$ of the image of $\omega_k$, with $\vec{v}$ in the column space of $M-I$, $\det(M-I) \not \equiv 0 \pmod{2^k}$, and $(\vec{v},M) \in \omega_k\left(\left[ \frac{K_k/\mathbb{Q}}{p} \right]\right)$. Then $(M-I)\vec{x} = -\vec{v}$ for some $\vec{x}$, which implies that $M\vec{x} +\vec{v} = \vec{x}$. Thus, there is a vector $\vec{x}$ fixed by the action of $(\vec{v},M)$.
By the above bijection, we know that the fixed point $\vec{x}$ corresponds to exactly one $\beta_{k(i)}$, and that the pre-image of $(\vec{v},M)$, an element of $\left[ \frac{K_k/\mathbb{Q}}{p} \right]$, fixes it.


Thus, the existence of a solution to the equation $(M-I)\vec{x} = -\vec{v}$ implies that the coordinates of some $\beta_{k(i)}$ are fixed in $K_k$ by $\sigma_p$. By definition of the Artin symbol, this implies that $\sigma_p$ fixes the coordinates of $\beta_{k(i)} \pmod{\mathfrak{p}}$. Since $\sigma_p$ is the Artin Symbol of $\Gal(K_k/\mathbb{Q})$ over a prime ideal, it is in fact the Frobenius automorphism, defined by $\sigma_p(x) \equiv x^p \pmod{\mathfrak{p}}$. Furthermore, the Frobenius automorphism is a generator of $\Gal\left(\frac{\mathcal{O}_{K_k}/ \mathfrak{p}}{\F_p}\right)$. Thus, $\beta_{k(i)}$ is fixed by all of $\Gal\left(\frac{\mathcal{O}_{K_k}/ \mathfrak{p}}{\F_p}\right)$, and we conclude that $\beta_{k(i)}$ is $\F_p$-rational.

Finally, the classical theory of Galois representations tells us that $\det(M-I)$ is equal to the order of $E(\mathbb{F}_p) \pmod{2^k}$. Now $\det(M-I) \not\equiv 0 \pmod{2^k}$, so the order of $E(\mathbb{F}_p)$ is not congruent to $0 \pmod{2^k}$. Recall that if $a$ is an element of a group, $order(na) = \frac{order(a)}{\gcd(order(a),n)}$. Since the order of $\beta_{k(i)}$ divides the order of the group, it follows that $order(2P)=order(2^{k+1}\beta_{k(i)})=order(P)$. We conclude that $P$ has odd order.

Thus, $P$ has odd order in $E(\mathbb{F}_p)$ for an unramified integer prime $p$ precisely when $-\vec{v}$ is in the column space of $(M-I)$, where $(\vec{v},M) \in \im(\omega_k)$, $\det(M-I) \not \equiv 0 \pmod{2^k}$, and $(\vec{v},M) \in \im\left(\left[ \frac{K_k/\mathbb{Q}}{p} \right]\right)$.\

\end{proof}


\section{Kinetic Subgroups of $\agl{k}$}
The work so far suggests that understanding the image of the arboreal representation will help to relate $P$ with fields $\mathbb{F}_p$ for which
$P \in E(\F_{p})$ has odd order. This becomes a group theory question concerning the affine general linear groups. We present a system of subgroups here that help dissect the image. As preliminaries, we use the following terminology.

\begin{definition} Let a subgroup $G\subset \agl{k}$ be called kinetic if both the following projection maps are surjective: $$\pr:G\twoheadrightarrow \gl{k},$$ $$\phi:G\twoheadrightarrow\agl{}.$$
\end{definition}

\begin{definition} A calculation in Magma shows that exactly one proper subgroup $G\subset \agl{2}$ is kinetic, up to conjugacy; let $H_2$ denote a specific representative of the conjugacy class. Let $\varphi_k:\agl{k}\rightarrow\agl{k-1}$ be the canonical projection from $\agl{k}$ down to $\agl{k-1}$ for $k\geq 2$. Then for $k\geq 3$, $H_k$ is recursively defined as 
$$H_k:=\varphi_k^{-1}(H_{k-1}).$$ 
For future use, it should be noted Magma also shows $H_3$ to be the only proper kinetic subgroup of $\agl{3}$ (up to conjugacy).\end{definition}

For reference, a set of generators for $H_2$ are 
$$\left\{\left(\begin{bmatrix}1\\2\end{bmatrix},\begin{bmatrix}2&1\\3&0\end{bmatrix}\right),\left(\begin{bmatrix}3\\3\end{bmatrix},\begin{bmatrix}2&3\\1&3\end{bmatrix}\right)\right\}.$$
Furthermore, $H_2$ is an index $4$, maximal subgroup of $\agl{2}$.\footnote{For computational purposes, it is easier to use a representation of $\AGL_2(\mathbb{Z}/n\mathbb{Z})$ embedded in $\GL_3(\mathbb{Z}/n\mathbb{Z})$ with the bijection $(\vec{v},M)\leftrightarrow\left[\begin{smallmatrix}M&\vec{v}\\0&1\end{smallmatrix}\right]$.}\\

\noindent Note that the kinetic conditions are given by surjectivity of maps. A consequence is that if a group surjects onto a kinetic group, it tends to inherent kinetic conditions. We make the claim that the only kinetic subgroups of $\agl{k}$ for $k\geq 2$ are exactly $H_k$ (up to conjugacy) and the whole group $\agl{k}$ itself. One direction is given here:

\begin{lemma} Both $\agl{k}$ and $H_k$ are kinetic subgroups of $\agl{k}$.\end{lemma}
\begin{proof} It is clear that both maps $\phi:\agl{k}\rightarrow\agl{}$ and $\pr:\agl{k}\rightarrow\gl{k}$ are both always surjective. 

For $H_k$, the claim is true by induction. By definition, $H_2$ was defined to be a proper kinetic subgroup of $\agl{2}$. For all $k>2$, since $H_k$ was defined as the pre-image of $H_{k-1}$, it is clear that both $H_k$ projects onto all of $\gl{k}$ since $H_{k-1}$ has second component intersecting all of $\gl{k-1}$ and also that $H_k$ surjects onto $\agl{}$.\end{proof}

To show the converse, namely, that $H_k$ is the only proper kinetic subgroup of $\agl{k}$, it turns out to be sufficient to examine only the case $k=3$ by Lemma ~\ref{BigLemma}, which will follow from some standard preliminary facts. (Recall that the Frattini subgroup $\Phi(G)$ of group $G$ is the intersection of all maximal subgroups of $G$.)

\begin{proposition}\label{frattini} Let $G$ be a group. Then the following are true, where $\Phi(\cdot)$ denotes the Frattini subgroup.
\begin{itemize}
\item Suppose $N\trianglelefteq G$ is a normal subgroup. Then $\Phi(N)\subset\Phi(G).$
\item Suppose $G$ is a $2$-group. Then for all $g$ in $G$, $g^2\in \Phi(G)$.
\end{itemize}
\end{proposition}

\begin{proposition} We have $|\gl{k}|=6\cdot 16^{k-1}$ and $|\agl{k}|=24\cdot 64^{k-1}$. \end{proposition}
\begin{lemma}\label{BigLemma}
Fix $k\geq 3$ and consider $r\in\mathbb{N}$ such that $3\leq r\leq k$. If $\exists(\vec{x},M)\in\agl{k}$ such that $(\vec{x},M)\equiv (\vec{0}, I)\pmod {2^r}$, then $(\vec{x},M)\in\Phi(\agl{k}).$\end{lemma}
\begin{proof}
Let $N_k:=\ker(\phi_k)$. Recall that $\phi_k:\agl{k}\twoheadrightarrow\agl{}$ is defined as the natural quotient map. And by the counting formula, 
$$|N_k|=\dfrac{|\agl{k}|}{|\agl{}|}=\dfrac{24\cdot 64^{k-1}}{24}=2^{6(k-1)}.$$
Thus, $N_k$ is a $2$-group.

We proceed by induction on the value $k-r$. For the case $r=k$, it is clear that $(\vec{0},I)\in \agl{k}$ and is in $\Phi(\agl{k})$.

\ih Suppose that for all $r\geq l+1$, the claim is true. Then examine some arbitrary element $(\vec{x}, M)\in \agl{k}$ such that 
$$(\vec{x},M)\equiv (\vec{0},I)\pmod {2^l}.$$
Thus, $(\vec{x},M)=(2^l\vec{y},I+2^{l}N)$ for some $\vec{y}\in(\mathbb{Z}/2^k\mathbb{Z})^2$ and $N\in\gl{k}$.

Note that $(2^{l-1}\vec{y},I+2^{l-1}N)\in \agl{k}$ as well. Then since $l\geq 3\Leftrightarrow 2(l-1)\geq l+1$,
\begin{align*}
(\vec{x},M)^2&=(2^{l-1}\vec{y},I+2^{l-1}N)^2\\
&=(2^{l-1}\vec{y}+2^{l-1}\vec{y}+2^{2(l-1)}N\vec{y},I+2^{l}N+2^{2(l-1)}N^2)\\
&=(2^{l}\vec{y}+2^{2(l-1)}N\vec{y}, I^2+2^{l}N+2^{2(l-1)}N^2)\\
&\equiv (2^l\vec{y},I+2^lN)\pmod {2^{2(l-1)}}\\
&\equiv (2^l\vec{y},I+2^lN)\pmod {2^{l+1}}.\end{align*} Equivalently,
$$(\vec{x},M)^2\cdot(2^l\vec{y},I+2^lN)^{-1}\equiv (\vec{0},I)\pmod {2^{l+1}}.$$
Therefore, by our inductive hypothesis we must have $(\vec{x},M)^2\cdot(2^l\vec{y},I+2^lN)^{-1}\in \Phi(\agl{k})$.

Note that $(2^{l-1}\vec{y},I+2^{l-1}N)^2\in\Phi(N_k)\subset\Phi(\agl{k})$ by proposition ~\ref{frattini}. Thus, closure implies that $(2^l\vec{y},I+2^lN)=(\vec{x},M)\in \Phi(\agl{k})$ as well, and so by induction we are finished.
\end{proof}

Note that any subgroup lies in a maximal subgroup. Because the kinetic conditions come from surjectivity of predefined maps, if a subgroup $H$ of $\agl{k}$ is kinetic, for all groups $G\subset\agl{k}$ containing $H$, $G$ is also kinetic. Thus, to find kinetic subgroups, we may in general examine maximal subgroups.

\begin{corollary}\label{BigCorollary} Suppose $M\subsetneq \agl{k}$ for $k\geq 3$ is a maximal, kinetic subgroup. Then $M$ and $H_k$ are conjugate subgroups. \end{corollary}
\begin{proof}
Let $M\subsetneq \agl{k}$ be a maximal, kinetic subgroup. Examine the map $\varphi:M\rightarrow\agl{3}$ by composition of the $\varphi_i$ maps. By Lemma~\ref{BigLemma}, every maximal subgroup of $\agl{k}$ contains $N_k$, the kernel. By the correspondence theorem, there is a bijection between subgroups containing the kernel and subgroups of the image. Since $\varphi$ is surjective, $M$ is in bijection with a maximal, proper subgroup of $\agl{3}$. Furthermore, it is clear that this subgroup of $\agl{3}$ is also kinetic because these maps preserve surjectivity. 

Recall by computation that conjugates of $H_3$ are the only maximal, kinetic subgroups of $\agl{3}$. Thus, $\varphi(M)=gH_3g^{-1}$, and by definition of the $H_k$ subgroups, $M\subset \varphi^{-1}(g)H_k\varphi^{-1}(g^{-1})$ is a subset of a conjugate of $H_k$. By assumption that $M$ is maximal, $M$ is a conjugate of $H_k$.
\end{proof}

Notice that this corollary followed from the fact that any kinetic maximal subgroup of $\agl{k}$ would have to show up in $\agl{3}$ by Lemma ~\ref{BigLemma}. In fact, we may make a similar statement about maximal subgroups of $H_k$.

We will later show that $|H_k|=6\cdot 64^{k-1}$. Since $H_k$ is kinetic, we have the following surjection:
$$\phi_k:H_k\twoheadrightarrow \agl{}.$$
Thus, by the counting formula we have $$|\ker(\phi_k)|_{H_k}|=\dfrac{|H_k|}{|\agl{}|}=\dfrac{6\cdot64^{k-1}}{24}=2^{6(k-1)-2},$$
and $\ker(\phi_k)|_{H_k}$ is a $2$-group.

\begin{lemma} Suppose $(2^l\vec{v},2^lM)\in H_k$. Then $(2^{l-1}\vec{v},2^{l-1}M)$ is also an element of $H_k$.
\end{lemma}
\begin{proof} We proceed by induction on $k$. For the case where $k=2$, a calculation in Magma verifies the claim. 

\ih Suppose that the claim is true for all $k\leq n\in\mathbb{N}$. Then examine an arbitrary $(2^l\vec{v},2^lM)\in H_{n+1}$. By the quotient map $\varphi_{n+1}$, we have $(2^l\vec{v},2^lM)\in H_{n}$. Then by our inductive hypothesis, we know that $(2^{l-1}\vec{v},2^{l-1}M)\in H_n$. By the definition of the $H_k$, since $\varphi_{n+1}((2^{l-1}\vec{v},2^{l-1}M))=(2^{l-1}\vec{v},2^{l-1}M)$, we must have $(2^{l-1}\vec{v},2^{l-1}M)\in H_{n+1}$. Therefore, by induction we are done.\end{proof}

Notice that every step of the proof of Lemma ~\ref{BigLemma} holds for $H_k$ as the main group, and so we similarly have the following corollary.
\begin{corollary}\label{SmallCorollary}For $k\geq 3$, there exists a bijection between maximal, kinetic subgroups of $H_k$ and maximal, kinetic subgroups of $H_3$. 
\end{corollary}
\begin{proof}
The natural mapping $\varphi:H_k\rightarrow H_3$ is the same as the map $\varphi:H_k\rightarrow \agl{3}$ because the $\varphi_i$ maps of $H_i$ are contained in $H_{i-1}$. Thus, similarly to Lemma~\ref{BigLemma}, all maximal subgroups of $H_k$ contain the kernel of $\varphi$. Therefore, there is a bijection between maximal kinetic subgroups of $H_k$ and maximal kinetic subgroups of $H_3$ as desired.
\end{proof}

This suffices to show the converse of the original claim.

\begin{theorem}\label{rum}
Let $k\geq 2$. The only kinetic subgroups of $\agl{k}$ are $\agl{k}$ itself and $H_k$.
\end{theorem}
\begin{proof} For $k=2$, we know from our initial definition of $H_2$ that $H_2$ was the only proper kinetic subgroup of $\agl{2}$.

By Corollary ~\ref{BigCorollary}, we know that for $k\geq 3$, the only maximal kinetic subgroups of $\agl{k}$ are $H_k$. Thus, we must check that no proper subgroups of $H_k$ are also kinetic. However, by Corollary ~\ref{SmallCorollary}, any kinetic proper subgroups of $H_k$ are in bijection with a proper kinetic subgroup of $H_3$. By Magma computation, we find that that there are no proper kinetic subgroups of $H_3$, hence there are no proper kinetic subgroups of $H_k$.
\end{proof}

\section{The Image of the Arboreal Representation}
The significance of these kinetic groups is that they are exactly the images of arboreal representations. We can say more for our specific curve and point.

\begin{lemma}\label{IsHk} For all $k\geq 2$, $\im(\omega_k)=H_k$.\end{lemma}
\begin{proof}

Our strategy is to show that $\im(\omega_k)$ must be kinetic but not the whole of $\agl{k}$, since by Theorem$~\ref{rum}$, the only two kinetic subgroups of $\agl{k}$ are the entire group and $H_k$.

To see that $\im(\omega_k)$ is kinetic, we may check the surjectivity conditions.
\begin{itemize}
\item 
The first condition for kinetic subgroups, that $\omega_k$ surjects onto $GL_{2}(\mathbb{Z}/2^k\mathbb{Z})$, can be checked using the criteria of \cite{MR2995149}. Their main theorem shows that $\omega_k$ surjects onto $GL_{2}(\mathbb{Z}/2^k\mathbb{Z})$ for all $k$ precisely when the curve $E$ satisfies the following conditions:
$ \left|\Delta(E)\right| \notin \mathbb{Q}^2, \left|\Delta(E)\right| \notin 2\mathbb{Q}^2$, $E$ has no rational 2-torsion, and there is no $t$ such that $j(E) = -4t^3(t+8)$.
The discriminant conditions can be checked by hand, and Magma easily determines that there are no rational solutions of the equation $j(E) = -4t^3(t+8)$. Their criteria then guarantee that $\im(\omega_k)$ satisfies the first kinetic condition.

\item We now show that the image of the representation satisfies the second kinetic condition.
Basic Galois theory applied to each coordinate of $Q$ implies that an element of $[2]^{-1}P$ is $\im(\omega_k)$-invariant if and only if there is a rational point $Q$ such that $2Q=P$.

Since $P=(4,7)$ is a generator of $E/\mathbb{Q}$, there is no rational point that doubles to $P$, and thus no element of $E[2] \cong \ztwok{k}^2$ that is fixed by $\im(\omega_k)\subset \agl{k}$. Note that $6$ divides $|\im(\omega_k)|$ since it surjects onto $\gl{}$, and that $|\im(\omega_k)|\geq 6$.
\begin{itemize}
\item Suppose $|\im(\omega_1)|=6$. In this case the image must be conjugate to $\{1\}\times\gl{}$, which we check in Magma. Consider first the case that the image is $\{1\}\times\gl{}$. Recall that the affine part of $\omega_1$ is defined by the mapping $\omega_1(\sigma) = \sigma(Q) - Q$. In particular, if the affine part of the image is trivial,  $\sigma(Q) - Q = 0$ for all $\sigma$ in our Galois group. But then $Q$ is a rational point that doubles to $P$, a contradiction. To take care of the conjugate cases we use that a subgroup fixes an element if and only if a conjugate subgroup fixes a similar element.
\item Suppose $|\im(\omega_1)|=12$. Note that $\agl{}\cong S_4$, and the only index $2$ subgroup is $A_4$. By calculation, an explicit representation of $A_4$ in $\agl{}$ does not surject onto $\gl{}$, a contradiction.
\item Thus, the only possible option is that $|\im(\omega_1)|=24$.
\end{itemize}
Therefore, we see that $\omega_k:\Gal(K_k/\mathbb{Q})\twoheadrightarrow\agl{}$ is surjective.
\end{itemize}

To see that $\im(\omega_k)$ is not all of $\agl{k}$, we use a criterion of Jones and Rouse. In Theorem 5.2 of \cite{MR2640290}, the authors show that an arboreal representation with surjective classical representation is surjective if and only if $P$ is not twice a rational point and $\Q(\beta_1) \nsubseteq \Q(E[4])$. We show that, in the case of our curve E, $\Q(\beta_1) \subseteq \Q(E[4])$.

Consider the tower of fields below.

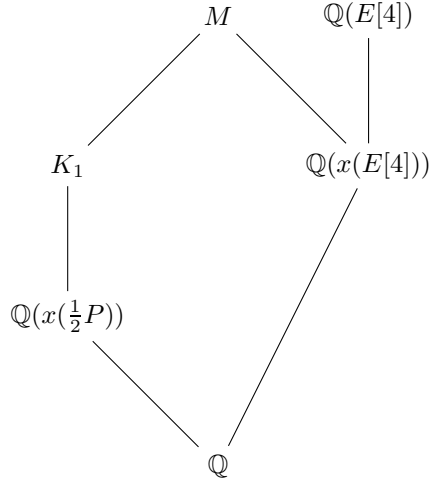
\begin{center}
\begin{tikzpicture}[node distance = 2cm, auto]
      \node (Q) {$\mathbb{Q}$};
      \node (XP) [above of=Q, left of=Q] {$\mathbb{Q}(x(\tfrac{1}{2}P))$};
      \node (xE4) [above of=Q, above of=Q, right of=Q] {$\mathbb{Q}(x(E[4]))$};
      \node (K) [above of=XP] {$K_1$};
      \node (E4) [above of=xE4] {$\mathbb{Q}(E[4])$};
      \node (M) [above of= K, right of=K] {$M$};
      \draw[-] (Q) to node {} (XP);
      \draw[-] (Q) to node {} (xE4);
      \draw[-] (XP) to node {} (K);
      \draw[-] (xE4) to node {} (E4);
      \draw[-] (K) to node {} (M);
      \draw[-] (xE4) to node {} (M);
\end{tikzpicture}
\end{center}

Using division polynomials, we may calculate the size of the Galois group of $\mathbb{Q}(x(E[4]))/\mathbb{Q}$, and thus the size of the extension itself. Indeed, Magma tells us that the degree of the extension is $48$. On the left side of the diagram, $K_1$ must have degree $24$ because $|\Gal(K_1/\Q)| = |\agl{}| = 24$. If we calculate the degree of $\mathbb{Q}(x(\tfrac{1}{2}P))$ using division polynomials, we get an extension of degree 24. Thus, the left side tower has collapsed, and the defining polynomial for each extension on the left hand side is actually just equal to the polynomial that defines the $x$-coordinates. Now we can determine whether or not $\Q(\beta_1) \subseteq \Q(E[4])$, since the smallest field containing both $\mathbb{Q}(x(E[4]))$ and $\mathbb{Q}(\tfrac{1}{2}P)$ will be the splitting field of the product of their respective defining polynomials. Magma tells us that this splitting field has degree $48$, and so the smallest field containing both is $\mathbb{Q}(x(E[4]))$. But $\mathbb{Q}(x(E[4])) \subseteq \mathbb{Q}(E[4])$.

Therefore, since $\im(\omega_k)$ is kinetic but not equal to $\agl{k}$, it must be equal to $H_k$.
\end{proof}

\noindent As defined above, $H_2$ is the non-trivial kinetic subgroup of $\agl{2}$, up to conjugacy. Since $H_k$ consists of lifts of $H_2$, the structure of $H_k$ is then determined by the structure of $H_2$. To understand $H_k$ well, we present some facts about $H_2$.

\begin{description}
\item[Size] By parity arguments, $|\gl{k}|$ is easily computed. It follows that $|\agl{k}|=24\cdot 64^{k-1}$. Computationally, we find that $[\agl{2}:H_2]=4$. Since $\varphi_k:H_k\rightarrow H_{k-1}$ is a surjective homomorphism, one has that $[\agl{k}:H_k]=[\agl{k-1}:H_{k-1}]$ for all $k\geq 3$. So by induction, we see that for all $k\geq 2$, $[\agl{k}:H_k]=4$ and
$$|H_k|=\dfrac{|\agl{k}|}{4}=6\cdot 64^{k-1}.$$

\item[Structure] We use the following terminology.
\begin{definition}
Fix $M\in\gl{k}$. Let $V_{M}:=\{\vec{a}:(\vec{a},M)\in H_k\}$. We call $V_{M}$ the set of associated vectors of $M$.
\end{definition}
\begin{definition}
For any $k\in\mathbb{N}$, define $$V_{i,j}:=\left\{\vec{x}\in\ztwok{k}:\vec{x}\equiv \left[\begin{smallmatrix}i\\j\end{smallmatrix}\right]\pmod 2\right\}.$$
\end{definition}
Because $H_2$ is finite and we have an explicit representation, computationally we find the following.

\begin{proposition} 
The group $H_2$ is the disjoint union of the sets $V_{0,0}\times J$, $V_{0,1}\times \left[\begin{smallmatrix}1&3\\0&1\end{smallmatrix}\right]J$, $V_{1,0}\times \left[\begin{smallmatrix}1&2\\0&1\end{smallmatrix}\right]J$, and $V_{1,1}\times \left[\begin{smallmatrix}1&1\\0&1\end{smallmatrix}\right]J$ where $J$ is a subgroup of $\gl{2}$ that is isomorphic to $\mathbb{Z}/3\mathbb{Z}\rtimes D_4$ by the action kernel $V_4$. 
\end{proposition}
\begin{proof}
Magma shows this to be true. For reference, $\left\{\left[\begin{smallmatrix}0&3\\1&0\end{smallmatrix}\right],\left[\begin{smallmatrix}1&3\\3&0\end{smallmatrix}\right]\right\}$ generates $J$.
\end{proof}
\end{description}


\section{Computing the Fraction}


Recall that if $p$ is a prime unramified in $K_k$, $(\vec{v},M) \in  \omega_k\left(\left[ \frac{K_k/\mathbb{Q}}{p} \right]\right)$, and $\det(M-I) \not\equiv 0 \pmod{2^k}$, then the point $P$ has odd order in $E(\mathbb{F}_p)$ if and only if $\vec{v}$ is in the column space of $(M-I)$.

By the Chebotarev density theorem,

\[
\lim_{x \to \infty} \frac{\left|\{ p~\text{prime and unramified in }K_k : p \leq x, \artin{K_k/\Q}{p} \subseteq S \}\right|}{\pi(x)} = \frac{|S|}{|\Gal(K_k/\Q)|},
\]
where $S$ a union of conjugacy classes in $\Gal(K_k/ \Q)$.
This is not quite the limit required to compute the fraction of primes dividing the ECHO sequence.
We will choose $S$ so that $\artin{K_k/\Q}{p} \subseteq S$ for all $k$ precisely when some $\sigma_p \in \artin{K_k/\Q}{p}$  fixes some $\beta_{k(i)}$.

We observe that if this fact is true for a certain $k$, it is true for $j \leq k$ as well. Thus, to find the primes for which that is true for all $k$ we need only consider its limit as $k$ approaches infinity.
Therefore,

$$\lim_{x \to \infty} \frac{\pi '(x)}{\pi (x)} =  
\lim_{k \to \infty} \lim_{x \to \infty} \frac{\left|\{ p~\text{prime and unramified in }K_k : p \leq x, \artin{K_k/\Q}{p} \subseteq S \}\right|}{\pi(x)}. $$

Moreover, since
 $$\frac{|S|}{|\Gal(K_k/ \mathbb{Q})|} = \frac{|\omega_k(S)|}{|\omega_k(\Gal(K_k/ \mathbb{Q}))|} = \frac{|\{ (\vec{v}, M) \in H_k : \vec{v} \in \im(M-I)\}|}{|H_k|},$$
we may compute the density in terms of the the structure of $H_k$.

We must make two separate choices of $S$ to compute this fraction.
We consider $\frac{|S_1|}{|\text{Gal}(K_k/ \mathbb{Q})|}$ where $S_1$ is the set of all elements in $\text{Gal}(K_k/ \mathbb{Q})$ such that
$$(\vec{v}, M) \in H_k \text{, det}(M-I) \not\equiv 0\pmod{2^k} \text{, and } \vec{v} \in \im(M-I).$$
This contains only elements corresponding to primes $p$ for which $P$ has odd order in $E(\mathbb{F}_p)$.
Secondly, we consider $\frac{|S_2|}{|\text{Gal}(K_k/ \mathbb{Q})|}$ where $S_2$ is the set of all elements in $\text{Gal}(K_k/ \mathbb{Q})$ such that for all $(\vec{v}, M) \in H_k$, $\vec{v}$ is in the column space of $M-I$.
This contains all the elements corresponding to primes $p$ for which $P$ has odd order in $E(\mathbb{F}_p)$, but not exclusively those elements. Both $S_1$ and $S_2$ are unions of conjugacy classes.

Consider $$\lim_{k \to \infty} \frac{|S_2| - |S_1|}{|\text{Gal}(K_k/ \mathbb{Q})|}
= \lim_{k \to \infty} \frac{|S_2 - S_1|}{|\text{Gal}(K_k/ \mathbb{Q})|}.$$

We then note
$$\lim_{k \to \infty} \frac{|S_2 - S_1|}{|\text{Gal}(K_k/ \mathbb{Q})|}
= \lim_{k \to \infty} \frac{ |\{ (\vec{v}, M) \in H_k \text{: det}(M-I) \equiv 0\pmod{2^k} \text{ and } \vec{v} \in \im(M-I) \}| }{|\text{Gal}(K_k/ \mathbb{Q})|}=0.$$

Thus, we may use $$\lim_{k \to \infty} \frac{|S_1|}{|\text{Gal}(K_k/ \mathbb{Q})|}= \lim_{k \to \infty} \frac{|S_2|}{|\text{Gal}(K_k/ \mathbb{Q})|}$$

to evaluate $\lim_{x \to \infty} \frac{\pi '(x)}{\pi(x)}$, since
$$\lim_{k \to \infty} \frac{|S_1|}{|\text{Gal}(K_k/ \mathbb{Q})|} \leq \lim_{x \to \infty} \frac{\pi '(x)}{\pi(x)} \leq \lim_{k \to \infty} \frac{|S_2|}{|\text{Gal}(K_k/ \mathbb{Q})|}.$$

For our computation of the fraction, we chose to evaluate the larger limit $$\lim_{k \to \infty} \frac{|S_2|}{|\text{Gal}(K_k/ \mathbb{Q})|} = \lim_{k \to \infty} \frac{|\{ (\vec{v}, M) \in H_k : \vec{v} \in \im(M-I)\}|}{|H_k|}. $$\\

Considering the above discussion, we make the following essential definition:
\begin{definition}
Let $M \in M_2(\mathbb{Z}/2^r \mathbb{Z})$, and $S \subseteq \mathbb{Z}/2^k \mathbb{Z}$. We define $$\mu(M, r, S) = \lim_{k \to \infty} \frac{|\{ (\vec{v}, M') : \vec{v} \in S, \vec{v} \in \im(M'), M' \equiv M \pmod {2^r} \}|}{|H_k|}.$$
\end{definition}
Note that if $r>1$, $M$ is invertible and $\im(M-I) \cap V_{M} \neq \emptyset$, we have the special case $$\mu(M-I,r,V_{0,0}) = \lim_{k \to \infty} \frac{|\{ (\vec{v}, M') \in H_k : \vec{v} \in \im(M'-I), M' \equiv M \pmod {2^r} \}|}{|H_k|}.$$

By partitioning $H_k$ at every level into lifts of matrices in $H_2$, the limit that we want to compute from the Chebotarev density theorem discussion becomes $$ \mu(H_2) \coloneqq \sum_{M \in GL_2(\mathbb{Z}/4 \mathbb{Z})} \mu(M-I, 2, V_M).$$ We first note that if $\im(M-I) \cap V_{M}$ is empty at level $k=2$, it must be empty for every lift. (Otherwise we could reduce the matrix and vector and get an element of the intersection at $k=2$.) Thus, we must only compute $\mu(M)$ for those $M$ where $\im(M-I) \cap V_{M}$ is nonempty at level $k=2$.


We now observe that
 $$\sum_{M \in GL_2(\mathbb{Z}/4 \mathbb{Z})} \mu(M-I, 2, V_M) = \lim_{k \to \infty} \sum_{M \in \gl{k}} \frac{|\im(M-I) \cap V_{M}|/|V_{M}|}{|\text{GL}_2(\mathbb{Z}/2^k \mathbb{Z})|}.$$
This fraction changes, however, depending on $k$ and on the determinant of $M-I$.\\

Therefore, instead of taking the sum of $\frac{|\im(M-I) \cap V_{M}|}{|V_{M}|}$ divided by $|GL_2(\mathbb{Z}/2^k \mathbb{Z})|$, we consider
$$\frac{|\im(M-I) \cap V_{M}|}{|V_{M}|} = \frac{|\im(M-I) \cap V_{M}|}{|\im(M-I)|} \frac{|\im(M-I)|}{|V_{M}|} $$
$$= \frac{|\im(M-I) \cap V_{M}|}{|\im(M-I)|} \cdot 4 |\det(M-I)|_2.$$

Here $|s|_{2}$ denotes the $2$-adic absolute values of $s$, namely,
$|s|_{2} = 2^{-{\rm ord}_{2}(s)}$. The second equality above follows from a result of \cite{Somos5}: if $\det(M)=2^r s$, with $s$ odd and $r < k$, $\frac{|\im(M)|}{|(\mathbb{Z}/2^k \mathbb{Z})^2|} = |\det(M-I)|_2 \coloneqq \frac{1}{2^r}$. From now on we define $ f_M \coloneqq \frac{|\im(M-I) \cap V_{M}|}{|\im(M-I)|}$.
Since $\omega_k$ is not surjective, the intersection of the column space of $M-I$ for a given $M \in \text{GL}_2(\mathbb{Z}/2^k \mathbb{Z})$  and the set of vectors $\vec{v} \in (\mathbb{Z}/2^k \mathbb{Z})^2$ such that $(\vec{v},M) \in H_k$ may be empty, or may be a fraction of $\im(M-I)$.  
Since $V_{i,j}$ contains one fourth of the vectors in $(\mathbb{Z}/2^k \mathbb{Z})^2$ for each $i,j$,the fraction the column space that contains a subset of $V_M$, the set of vectors for which $\{ (\vec{v},M)|\vec{v} \in V_M \}$ is in $H_k$, may take on one of a few values between $0$ and $1$.

We may easily determine the possible values for $f_M$ and divide the sum in the above limit into sub-sums by the following result.

\begin{lemma}{If $\im(M-I) \cap V_M \neq \emptyset$, then $|\im(M-I) \cap V_M| = |\im(M-I) \cap V_{0,0}|$. Therefore, $f_M$ may take the values $1, \frac{1}{2}, \frac{1}{4}, \text{ or } 0$.}
\begin{proof}
Suppose $\im(M-I) \cap V_M \neq \emptyset$.  Consider $\sigma : \im(M-I) \cap V_{0,0} \to \im(M-I) \cap V_M$ given by $\sigma(\vec{w}) = \vec{v} + \vec{w}$.
Any element of $\im(M-I) \cap V_M$ differs from a fixed $\vec{v} \in \im(M-I) \cap V_M$ by an element of $\im(M-I) \cap V_{0,0}$, including $\vec{v}$ itself.
Thus, $\sigma$ is a bijection, since $\vec{w_1} + \vec{v} = \vec{w_2} + \vec{v}$ implies that $\vec{w_1}=\vec{w_2}$ and given any $\vec{x} \in \im(M-I) \cap V_M$ we know $( \vec{x} - \vec{v}) \in \im(M-I) \cap V_{0,0}$.
Therefore, $|\im(M-I) \cap V_{0,0}| = |\im(M-I) \cap V_M|$.\\

By the second isomorphism theorem of modules, $[\im(M-I) : \im(M-I) \cap V_{0,0}] =   [\im(M-I) + V_{0,0} : V_{0,0}]$. Since $V_{0,0}$ has index four in the entire set of vectors, it has index two or one in $\im(M-I) + V_{0,0}$, and so we are done. Finally, if $\im(M-I) \cap V_M = \emptyset$, then $f_M = 0$.

\end{proof}
\end{lemma}

We next observe that the fraction $f_M$ for a given $M \in GL_2(\mathbb{Z}/ 4 \mathbb{Z})$ is the same for all lifts of $M$ in $\gl{k}$.

\begin{lemma}{The fraction $f_M$ equals $f_{M_k}$ for all lifts $M_k \in \gl{k}$ of $M \in \text{GL}_2(\mathbb{Z}/ 4 \mathbb{Z})$.}
\begin{proof}

Since $V_{0,0} \in (\mathbb{Z}/ 4 \mathbb{Z})^2$ lifts to $V_{0,0} \in (\mathbb{Z}/ 2^k \mathbb{Z})^2$ for any $k$ and because $M' \equiv M \pmod{4}$, the image of $M-I \in GL_2(\mathbb{Z}/ 4 \mathbb{Z})$ lifts to the image of $M'-I \text{ in }GL_2(\mathbb{Z}/ 2^k \mathbb{Z})$ we conclude $\im(M'-I) \cap V_{0,0}$ is a lift of $\im(M-I) \cap V_{0,0}$.  Therefore, the index of both intersections in the respective images remains $f_M$.

\end{proof}
\end{lemma}


We can now explicitly calculate the fraction using the following cases.

\begin{case}[$\det(M-I)$ is invertible]
\end{case}
Any lift of an invertible matrix is invertible. Indeed, for these matrices and all of their lifts, the image of $M-I$ is the entirety of $(\mathbb{Z}/2^k \mathbb{Z})^2$, so $\frac{|V_{0,0} \cap \im(M-I)|}{|V_0,0|} = 1$, and $\mu(M, 2, V_{0,0}) = \frac{1}{96}$. Since there are 32 matrices with invertible determinant, their total contribution to the sum is $\frac{32}{96} = \frac{1}{3}.$

\begin{case}[$\det(M-I) \equiv 2 \mod{4}$]
\end{case}

For all of the matrices with at least one associated vector in the image of $M-I$, of which there are 12 in $H_2$, $f_M = \frac{1}{2}$ and $|\det(M-I)|_2 = \frac{1}{2}$. The total contribution of these matrices is $\frac{12}{96} \cdot 4 \cdot \frac{1}{2} \cdot \frac{1}{2} = \frac{1}{8}$.

\begin{case}[$\det(M-I) = 0$ and $M-I$ has at least one odd entry]
\end{case}

For this case, we use two lemmas of \cite{Somos5}: First, given a matrix $M \in M_2(\mathbb{Z}/2^k \mathbb{Z})$ where $\det(M-I) = 0$ and $M-I$ has at least one odd entry, half of the lifts in $M_2(\mathbb{Z}/2^{k+1} \mathbb{Z})$ have determinant $2^k$ and half have determinant zero. Second, as we used before, if $\det(M)=2^r s$, with $s$ odd and $r < k$, $\frac{|\im(M)|}{|(\mathbb{Z}/2^k \mathbb{Z})^2|} = \frac{1}{2^r}$. We will also use the fact that for any lift $M'$ of $M-I$, $f_{M'}=f_{M}$. A computation in Magma shows that, for the matrices $M$ with at least one associated vector in the image of $M-I$, $f_M = \frac{1}{2}$.

Thus, for $M \in H_2$ with $\det(M-I)=0$ and with $M-I$ having at least one odd entry, $$\mu(M,2,V_{0,0}) = \frac{1}{96}  \lim_{n \to \infty} \sum_{i=2}^{n} 4f_M|\det(M-I)|_2\frac{1}{2^{i-1}} = \lim_{n \to \infty} \sum_{i=2}^{n} 4 \cdot \frac{1}{2} \cdot \frac{1}{2^{i}} \cdot \frac{1}{2^{i-1}} = \frac{1}{3},$$ where the sum corresponds to the fact that at level $i$, $\frac{1}{2^{i-1}}$ of lifts of $M$ have $|\det(M'-I)|_2 = \frac{1}{2^{i}} $. A Magma computation shows that there are 12 matrices $ M \in H_2$ that have at least one associated vector in the image of $M-I$. Therefore, the total contribution from these 12 matrices is $12 \cdot \frac{1}{96} \cdot \frac{1}{3} = \frac{1}{24}$.

\begin{case}[$\det(M-I) = 0$ and $M-I$ has all even entries]
\end{case}

\begin{lemma} Suppose $M$ has all even entries, and $r>1$. Then $$\mu(M,r,V_{0,0}) = \frac{1}{16} \mu\left(\frac{M}{2},r-1,(\mathbb{Z}/2^{r-1} \mathbb{Z})^2 \right).$$ Further, when $r=1$, $\mu\left(M,1,(\mathbb{Z}/2 \mathbb{Z})^2 \right) = \frac{1}{64} \mu\left(\frac{M}{2},0,\{0\}\right).$
\end{lemma}
\begin{proof}
We have $\vec{v} \in \im(M)$ if and only if $\frac{\vec{v}}{2} \in \im\left(\frac{M}{2}\right)$. Thus, the numerators of the two limits in $\mu(M,r,V_{0,0})$ and $\mu\left(\frac{M}{2},r-1,(\mathbb{Z}/2^{r-1} \mathbb{Z})^2\right)$ are equal, but reducing $r$ by one reduces the denominator of $\mu\left(\frac{M}{2},r-1,(\mathbb{Z}/2^{r-1} \mathbb{Z})^2 \right)$ by $64$. The change of the ambient set of vectors then gives the extra factor of four. The second case has a factor of $\frac{1}{64}$ because we do not readjust for the ambient set of vectors. 
\end{proof}

If $M-I$ has all even entries, clearly $\im(M-I) \subseteq V_{0,0}$. Therefore, we need only compute $(M-I)/2$ for $M$ in $J_1$, since $\im(M-I) \cap V_{M}$ will be empty for $M$ in all other cosets. For three matrices in $J_1$, $(M-I)/2$ is invertible. For these three, as in the invertible case with $r=2$, an analogous computation to Case 1 gives $\mu\left(\frac{M-I}{2},1,(\mathbb{Z}/2 \mathbb{Z})^2\right) = \frac{1}{|\gl{}|} = \frac{1}{6}$. By our lemma, each contributes to the total sum at level $r=2$ an amount of $\mu(M, 2,V_{0,0}) = \frac{1}{16}\mu\left(\frac{M-I}{2}, 1,(\mathbb{Z}/2 \mathbb{Z})^2\right) = \frac{1}{16} \cdot \frac{1}{6}$. Since there are three such matrices, their total contribution is $\frac{1}{16} \cdot \frac{1}{6} \cdot 3 = \frac{1}{32}$.

For only one of them (the identity matrix), $(M-I)/2 = \left( \begin{smallmatrix} 0&0\\ 0&0 \end{smallmatrix} \right)$. In this case, $\mu(\left( \begin{smallmatrix} 0&0\\ 0&0 \end{smallmatrix} \right), 2, V_{0,0}) =\frac{1}{16} \mu\left(\left( \begin{smallmatrix} 0&0\\ 0&0 \end{smallmatrix} \right), 1,(\mathbb{Z}/2 \mathbb{Z})^2\right)$ since $M-I$ has all even entries. Now we apply our lemma again: $$\mu\left(\left( \begin{smallmatrix} 0&0\\ 0&0 \end{smallmatrix} \right), 1,(\mathbb{Z}/2 \mathbb{Z})^2\right) = \frac{1}{64} \mu\left(\left( \begin{smallmatrix} 0&0\\ 0&0 \end{smallmatrix} \right), 0,\{0\}\right) = \frac{1}{64} \sum_{M \in M_2(\mathbb{Z}/2 \mathbb{Z})} \mu\left(M,1,(\mathbb{Z}/2 \mathbb{Z})^2\right).$$ The last equality comes from the consideration that all matrices at level $k=1$ are lifts of the unique matrix modulo 1. Continuing, 
$$\frac{1}{64} \cdot \sum_{M \in M_2(\mathbb{Z}/2 \mathbb{Z})} \mu\left(M,1,(\mathbb{Z}/2 \mathbb{Z})^2\right) = \frac{1}{64} \mu\left(\left( \begin{smallmatrix} 0&0\\ 0&0 \end{smallmatrix} \right), 1,(\mathbb{Z}/2 \mathbb{Z})^2\right) + \frac{1}{64} \sum_{\substack{M \in M_2(\mathbb{Z}/2 \mathbb{Z}),\\ M \neq \left( \begin{smallmatrix} 0&0\\ 0&0 \end{smallmatrix} \right)}} \mu\left(M,1,(\mathbb{Z}/2 \mathbb{Z})^2\right).$$

From here we may solve for our unknown:$$\frac{63}{64} \mu\left(\left( \begin{smallmatrix} 0&0\\ 0&0 \end{smallmatrix} \right), 1,(\mathbb{Z}/2 \mathbb{Z})^2\right) = \frac{1}{64} \sum_{\substack{M \in M_2(\mathbb{Z}/2 \mathbb{Z}),\\ M \neq \left( \begin{smallmatrix} 0&0\\ 0&0 \end{smallmatrix} \right)}} \mu\left(M,1,(\mathbb{Z}/2 \mathbb{Z})^2\right).$$ Of the matrices appearing in the right hand side of the above sum, nine have even determinant and at least one nonzero entry, and six have odd determinant. Those six contribute $\frac{1}{6}$ as before, and an identical infinite sum argument as in case three gives a $\mu$ value of $\frac{1}{18}$ for the remaining nine. So $\mu\left(\left( \begin{smallmatrix} 0&0\\ 0&0 \end{smallmatrix} \right), 1,(\mathbb{Z}/2 \mathbb{Z})^2\right) = \frac{63}{64}\left(6 \cdot \frac{1}{6}+9 \cdot \frac{1}{18}\right) = \frac{1}{42}$.
We thus conclude that, at level $r=2$, the identity matrix contributes a $\mu$ value of $\frac{1}{16} \cdot \frac{1}{42} = \frac{1}{672}$.

Since these four cases partition the elements of $H_2$, $\mu(H_2) = \frac{1}{3} + \frac{1}{24} + \frac{1}{8} + \frac{1}{32} + \frac{1}{672} = \frac{179}{336}$. This concludes our proof of Theorems~\ref{main} and ~\ref{main2}.

\section{A Family of Elliptic Curves with Arboreal Representation $H_k$}
In proving the original fact concerning the density of primes dividing the ECHO sequence, we have found that for an elliptic curve and rational point with certain conditions, the arboreal representation has image conjugate to $H_k$, the only proper kinetic subgroup of $\agl{k}$ (up to conjugacy). These $H_k$ subgroups are especially interesting because they only arise from looking at arboreal representations modulo a power of $2$. In \cite{MR2640290}, the authors showed that the arboreal representations for primes $\ell$ larger than $2$ (with similar condition that no rational $\ell$-division point exists) always surject onto $AGL_{2}(\mathbb{Z}/\ell^k\mathbb{Z})$ if the classical Galois representations are surjective. We thus ask whether the exceptional subgroups $H_k$ appear as images of the arboreal representations for infinitely many pairs $(E,P)$.

To answer this question, we examine a certain family of elliptic curves to find a parametrization of elliptic curves and points that have the same arboreal image as our original curve. Keep in mind the conditions that we must show these curves obey:
\begin{itemize}
\item $E/\mathbb{Q}$ is an elliptic curve with a point $P$ such that $P$ has no rational $2$-division point.
\item The classical map $\rho_k:\Gal(K_k/\mathbb{Q})\rightarrow \gl{k}$ is surjective.
\item The arboreal map $\omega_k:\Gal(K_k/\mathbb{Q})\rightarrow \agl{k}$ is not surjective.
\end{itemize}

We examine the family of curves $E_{a,b}:y^2+axy+by=x^3+bx^2$ with the point $(0,0)$. By exercise $3.1$ of Silverman \cite{MR1312368}, any elliptic curve $E/K$ with a point $P\in E(K)$ such that $P,2P,3P\neq 0$ has an equation of the form of $E_{a,b}$ with $P$ corresponding to $(0,0)$. We will later take $E$ such that $P$ does not have any rational $2$-division points and such that the classical Galois representation attached to $E$ is surjective. These conditions will ensure that $P,2P,3P\neq 0$.

To find curves with arboreal image equal to $H_k$, we first construct a surface that parametrizes pairs $(E,P)$ with arboreal image contained in $H_k$ (or sufficiently, $H_2$). This result is summarized as follows.
\begin{lemma}
Let $E_{a,b}/\mathbb{Q}$ be the elliptic curve $y^2 + axy + by = x^3 + bx^2$ such that the $j$-invariant is not equal to $\frac{2048t^3}{3(t+1)(t+3)^2}$ for some $t\in\mathbb{Q}$. Furthermore, let $P=(0,0)$ be on $E_{a,b}$, $K_2 = \mathbb{Q}([4^{-1}]P)$, and $\omega_2: \Gal(K_2/\mathbb{Q}) \mapsto \AGL_2(\mathbb{Z}/4\mathbb{Z})$. Then $\im(\omega_2)$ is conjugate to a subgroup of $H_2$ if and only if the ``fabulous'' polynomial

\begin{align*}\label{mylab}
f_{a,b}(x) =&\mbox{ }x^4-768b^2 x^3-2048(a^4 b^3 + a^3 b^3 - 8a^2 b^4 + 36a b^4 - 16b^5 + 81b^4)x^2 \\
 &+ 1048576(a^4 b^5 - a^3 b^5 + 8a^2 b^6 - 36a b^6 + 16b^7)x \\
 &+ 262144(-a^{10} b^5 + a^9 b^5 - 16a^8 b^6 + 72a^7 b^6 - 96a^6 b^7 - 55a^6 b^6 + 512a^5 b^7 \\
 &- 256a^4 b^8 - 1724a^4 b^7 + 896a^3 b^8 + 1272a^3 b^7 - 256a^2 b^9 \\
 &- 3984a^2 b^8 - 256a b^9 + 18144a b^8 - 8256 b^9 - 8748b^8)
\end{align*}

has a rational root.
\end{lemma}
\begin{proof}
Regard $a$ and $b$ as indeterminates and $E_{a,b}$ as a curve over $K=\mathbb{Q}(a,b)$. Let $\beta_{2}$ be a 4-division point of $P$; all $4$-division points of $P$ have the form $\beta_2 + cA + dB$ of $P$, where $0 \leq c,d \leq 3$, and $A, B$ generate $E[4] \cong (\mathbb{Z}/4\mathbb{Z})^2$. For each of the four left cosets $C$ of $H_2$ in $\AGL_2(\mathbb{Z}/4\mathbb{Z})$, define
$$\zeta_C = \sum\limits_{\sigma \in C} \sigma(x(\beta_2+A))\cdot  \sigma(x(\beta_2+B))\cdot  \sigma(x(\beta_2 + A + B)).$$
Now for $\sigma \in \Gal(K_2/\mathbb{Q}) \subset \AGL_2(\mathbb{Z}/4\mathbb{Z})$ we have $\sigma(\zeta_C) = \zeta_{\sigma C}$, and hence the polynomial
$$f_{a,b}(x) := \prod\limits_{C} (x-\zeta_C)$$
has coefficients in $K$. We computed the degree $16$ polynomial whose roots are the
$x$-coordinates of the $4$-division points of $P$ and found that the coefficients are in $\Z[a,b]$. This implies that the $\zeta_{C}$ are integral over $\Z[a,b]$. Since $\Z[a,b]$ is integrally closed in $K$, the coefficients of $f_{a,b}(x)$ are integral over $K$ and hence are in $\Z[a,b]$.

If we imagine our curve as $E_{a,b,c}:y^2+axy+bcy=x^3+bx^2$, making the change of variables $x \mapsto s^{2} x$ and $y \mapsto s^{3} y$, the coefficients rescale as $a\mapsto sa$, $b\mapsto s^2b$ and $c\mapsto sc$. (We say that $a$ and $c$ have weight $1$, while $b$ has weight $2$.) This shifts the roots of an analogous $f_{a,b,c}$ from $\zeta\mapsto s^6 \zeta$ (since each $\zeta$ has three factors of $x$-coordinates). It follows that the weight of the coefficient of $x^{i}$ in $f_{a,b,c}$ is $24-6i$. Specializing to $c=1$, we have an upper bound on the degree of the coefficients of $f_{a,b}$.

To numerically compute $f_{a,b}(x)$, we determine how automorphisms of $K_2$ act on the $x$-coordinates. Let $\sigma \in \Gal(K_2/\mathbb{Q})$ and let $\omega_2(\sigma) = (\vec{v},M)$ where $\vec{v}$ is a vector in ($\mathbb{Z}/4\mathbb{Z}$) and $M \in AGL_2(\mathbb{Z}/4\mathbb{Z})$. Then $\sigma(\beta_2 + cA + dB) = \beta_2 + eA + fB$ where $e, f$ are the vector coordinates computed by the equation $\begin{bmatrix}e\\f\end{bmatrix} = M\begin{bmatrix}c\\d\end{bmatrix} + \vec{v}.$

We created a function $\sigma$ in Magma that produces the $x$-coordinate of $\beta_2 + eA + fB$ based on the input $\begin{bmatrix}c\\d\end{bmatrix}$ by using the periods of $E$ and the elliptic logarithm and elliptic exponential functions.
We computed $f_{a,b}$ for many values of $a$ and $b$ and then determined the coefficients of $f_{a,b}$ as polynomials in $a$ and $b$ using linear algebra and the degree bound proven above.

The discriminant of $f_{a,b}$ is $-b^{15}\Delta_{a,b} g(a,b)$ with $\Delta$ the discriminant of the curve $E_{a,b}$ and $$g(a,b)=a^9 + 16a^7b - 46a^6b + 96a^5b^2 - 360a^4b^2 + 256a^3b^3 + 512a^3b^2 - 672a^2b^3 + 256ab^4 + 128b^4.$$
In fact, $g(a,b)$ defines a curve of genus $0$, and this curve is isomorphic to $\mathbb{P}^{1}$.
Elliptic curves defined by $a$ and $b$ on this bad locus have $j$-invariant equal to $\frac{2048\cdot t^3}{3(t+1)(t+3)^2}$. We do not consider this case, and so we may assume that $g(a,b)\neq 0$. Similarly, $\Delta_{a,b}$ or $b=0$ make $E_{a,b}$ singular. Thus, we may assume the discriminant is non-zero and $f_{a,b}$ has distinct roots. 

Following the argument of \cite{MR2995149}, we have that $\sigma\in\Gal(K_2/\Q)$
acts on the roots of $f_{a,b}$ by $\sigma(\zeta_{C}) = \zeta_{\sigma C}$. Since the roots are
distinct, we have that $f_{a,b}$ has a rational root if and only if for all $\sigma \in \Gal(K_{2}/\Q)$, there is a coset $C$ so that $\sigma C = C$. This is equivalent to the image of $\omega_{2}$ being contained in a conjugate of $H_{2}$.
\end{proof}

If we consider $f_{a,b}(x)$ as a polynomial in three variables $f(x,a,b)$, then there is a rational curve lying on the surface $f(x,a,b)=0$ which yields a one-parameter family of curves $E_{a(t),b(t)}$. Indeed, the polynomial $f(-96b^2,a,b)$ is a curve of genus 0 with a rational point, which implies that the points on the curve can be parametrized. Using Magma, one finds that a parametrization is given by $(p_1(t):p_2(t):p_3(t))$, where
\begin{align*}
p_1(t) =& (t-25)(t+35)(t^2-29t+676)(t^2-10t-279)^2(t^2+10t+97),\\
p_2(t) =& (t-25)(t+35)^2(t^2-29t+676)^3,\\
p_3(t) =& (t^2-10t-279)^4(t^2+10t+97).
\end{align*}
Thus, rational points on $f(-96b^2,a,b)$ are of the form $(a,b) = \left(\frac{p_1(t)}{p_3(t)},\frac{p_2(t)}{p_3(t)}\right)$.

Having computed this family, we proceed with the proof of our final theorem, restated here:

\begin{main3}
There exists a one-parameter family of curve-point pairs $(E,P)$ such that infinitely many pairs in the family have arboreal image conjugate to $H_k$. In particular, infinitely many pairs $(E,P)$ in our family have the property that the density of primes $p$ for which the reduction of $P$ modulo $p$ has odd order is $\frac{179}{336}$.
\end{main3}

\begin{proof}
We invoke Hilbert's Irreducibility Theorem to show that the image of the arboreal representation is kinetic for curves in our one-parameter family. The conditions for a subgroup to be kinetic, as described in Section 6, are all irreducibility conditions on various polynomials. First, the conditions for surjectivity of the classical Galois representation given in \cite{MR2995149} are defined by the irreducibility over $\Q$ of the 2-torsion polynomial, $x^2-\Delta_t$, $x^2+\Delta_t$, $x^2-2\Delta_t$, $x^2+2\Delta_t$, and $j_t+4x^3(x+8)$.  This irreducibility also ensures that the curves are not in the bad locus of our surface, since it can be checked that all curves in the bad locus have rational 2-torsion. The second kinetic condition is satisfied if the point $(0,0)$ is not twice a rational point, which is ensured if the 2-division polynomial of $(0,0)$ is irreducible.

Now a Magma computation shows that all of these polynomials are irreducible over $\Q(t)$ for the $\Delta_t$ and $j_t$ associated to our family of curves $E_t$. Hilbert's Irreducibility Theorem then implies that there are infinitely many specializations of $t$ such that all of the required polynomials stay irreducible over $\Q$ simultaneously. Furthermore, neither the $j$-invariant nor the discriminant of our family are constant. We conclude that there are infinitely many non-isomorphic and non-singular elliptic curves in our family that have as $\im(\omega_k)$ a kinetic subgroup of $\agl{k}$.

Since the image of the arboreal representation is kinetic for infinitely many curves in this family, and every rational point on $f(x,a,b) = 0$ corresponds to a pair $(E,P)$ with arboreal representation contained within $H_2$, we deduce that these infinitely many curves in fact have image equal to $H_k$ for all $k$. This concludes the proof of Theorem~1.4.

\end{proof}

\appendix
\section{A second look at the ECHO sequence} 
Recall that the ECHO sequence was itself motivated by finding a sequence whose elliptic curve partner has special Galois-theoretic properties. The inspiration behind this first definition (presented in the abstract and duplicated here) was to mirror the general form of the Somos-4 sequence that inspired the work of Rouse and Jones' paper. We hence derived an analogous family of curves. However, proceeding in this ad hoc manner we have only a clunky definition that gives no hints as to why the sequence is structured the way it is; for example, why does this definition have different recurrence relations modulo 3? 

We present here a second definition of the ECHO sequence that is perhaps a little more mathematically pleasing, and also gives better insight into the structure of the sequence by providing a more involved recursive definition that might make any properties derived from the sequence better explained. Furthermore, to flesh out a more whole picture of the sequence, we prove the main result of section~\ref{sequences} using the added flexibility of the two equivalent definitions.

First recall our original definition.
\begin{definition}
We define the ECHO sequence $\{b_n\}$ recursively by $(b_0,b_1,b_2,b_3)=(1,1,2,1)$ and for $n\geq 4$,
$$b_n=\begin{cases}
\dfrac{b_{n-1}b_{n-3}-b_{n-2}^2}{b_{n-4}} &\mathrm{if}~ n\not\equiv 0\pmod 3,\\
\dfrac{b_{n-1}b_{n-3}-3b_{n-2}^2}{b_{n-4}} &\mathrm{if}~ n\equiv 0\pmod 3.
\end{cases}$$
\end{definition}

However, a second look upon our work offered to us by Michael Somos lends us another perspective.

\begin{definition} Define the ECHO sequence $\{b_n\}$ by $$(b_0,b_1,b_2,b_3,b_4,b_5,b_6)=(1, 1, 2, 1, -3, -7, -17),$$ and further recursively for $n\geq 7$ by $$b_n:=\dfrac{-b_{n-6}b_{-1}+5b_{n-4}b_{n-3}}{b_{n-7}}.$$

And again, using the relationship $b_{n}b_{n+7}=-b_{n+1}b_{n+6}+5b_{n+3}b_{n+4}$, the second definition also defines the ECHO sequence for negative indices.\end{definition}

The flexibility of this new definition is that it need not be concerned with indexing; the use of more previous elements of the sequence, while more ``restrictive'' in some sense, also gives us more ability to manipulate coprimality conditions.

To prove that the definitions are equivalent, recall this lemma from section~\ref{sequences} and consider the following derived sequence that will bridge the gap between the two definitions. 

\begin{lemma}
For $n\geq 0$, define
$$h(n)=\begin{cases}
b_{n-3}^2b_{n}^2 + b_{n-3}b_{n-1}^3 + 3b_{n-2}^3b_{n} - 3b_{n-2}^2b_{n-1}^2 & \mathrm{if}~n\equiv 0\pmod 3,\\
3b_{n-3}^2b_{n}^2 + b_{n-3}b_{n-1}^3 + b_{n-2}^3b_{n} - b_{n-2}^2b_{n-1}^2 &  \mathrm{if}~n\equiv 1\pmod 3,\\
b_{n-3}^2b_{n}^2 + 3b_{n-3}b_{n-1}^3 + b_{n-2}^3b_{n} - 3b_{n-2}^2b_{n-1}^2 & \mathrm{if}~n\equiv 2\pmod 3.
\end{cases}$$
Then $h(n)=0$ for all $n\in\mathbb{N}$.
\end{lemma}

\begin{definition} For all $n\in\mathbb{N}$, define the sequence $d_n$ by $d_n:=b_nb_{n+5}-b_{n+2}b_{n+3}$. \end{definition}

We truly mean that $\{d_n\}$ is well-defined for either definition $\{b_n\}$ of the ECHO sequence. However, before having proved equality, one may consider there to be two sequences $d_{1,n}$ and $d_{2,n}$ for each definition of the ECHO sequence. This is of short lived consequence to us, however, because we directly prove their equivalence here.

\begin{corollary} For all $n\in\mathbb{N}$, $b_{n+7}d_{1, n}=b_{n+1}d_{1, n+3}$.\end{corollary}
\begin{proof}
For all $k\in\mathbb{N}$, we hope to show
\begin{align*}b_{k+7}d_k&=b_{k+1}d_{k+3}\\
\Leftrightarrow b_{k}b_{k+5}b_{k+7}+b_{k+1}b_{k+5}b_{k+6}-b_{k+2}b_{k+3}b_{k+7}-b_{k+1}b_{k+3}b_{k+8}&=0.\end{align*}
Again after rewriting this expression with terms $b_{k}$ through $b_{k+4}$ with substitution by our (first) definition, factoring with Magma we find that $h(k+3)=0$ is a factor. And equivalently, $b_{k+2}d_k=b_{k-4}d_{k+3}$ as desired.
\end{proof}

\begin{proposition} For all $n\in\mathbb{N}$, $b_{n+7}d_{2,n}=b_{n+1}d_{2, n+3}$.\end{proposition}
\begin{proof}
By substitution, we seek to show that for $n\in\mathbb{N}$, 
$$b_nb_{n+5}b_{n+7}-b_nb_{n+2}b_{n+3}=b_{n+1}b_{n+3}b_{n+8}-b_{n+1}b_{n+5}b_{n+6}.$$

By substitutions with our recursive (second) definition, we have the following: \begin{align*}
-b_{n+1}b_{n+5}b_{n+6}-b_nb_{n+2}b_{n+3}&=-b_{n+2}b_{n+3}b_{n+7}-b_{n+1}b_{n+5}b_{n+6}\\
-b_{n+1}b_{n+5}b_{n+6}+5b_{n+3}b_{n+4}b_{n+5}-b_nb_{n+2}b_{n+3}&=-b_{n+2}b_{n+3}b_{n+7}+5b_{n+3}b_{n+4}b_{n+5}-b_{n+1}b_{n+5}b_{n+6}\\
b_{n+5}(-b_{n+1}b_{n+6}+5b_{n+3}b_{n+4})-b_nb_{n+2}b_{n+3}&=b_{n+3}(-b_{n+2}b_{n+7}+5b_{n+4}b_{n+5})-b_{n+1}b_{n+5}b_{n+6}\\
b_nb_{n+5}b_{n+7}-b_nb_{n+2}b_{n+3}&=b_{n+1}b_{n+3}b_{n+8}-b_{n+1}b_{n+5}b_{n+6}.\end{align*} 
Thus, we are done.\end{proof}

We note that both definitions of the ECHO sequence satisfy the equation 
$$b_{k}b_{k+5}b_{k+7}+b_{k+1}b_{k+5}b_{k+6}=b_{k+2}b_{k+3}b_{k+7}+b_{k+1}b_{k+3}b_{k+8},$$
and since we may compute the first few values, induction shows both definitions produce identical sequences as desired. Moreover, this derived sequence $\{d_n\}$ reveals even stronger statements about the ECHO sequence.

\begin{corollary} For all $n>1$,
$$d_n=\begin{cases}
b_{n+1}b_{n+4}&\text{if }n\equiv 0,1\pmod 3\\
3b_{n+1}b_{n+4}&\text{if }n\equiv 2\pmod 3.
\end{cases}
$$\end{corollary}
\begin{proof}
Examine first the case $n\equiv 0\mod 3$. Computationally. $d_3=b_3*b_8-b_5*b_6=(-3)(247)-(-17)(2)=-707=(-7)(101)$ so the claim is true for $n=3$. For the sake of induction, suppose that all $n\leq k$ satisfy the claim for some $k\equiv 0\mod 3$. Then 
\begin{align*}
b_{k+1}b_{k+4}&=d_k\\
b_{k+1}b_{k+4}&=\dfrac{b_{k+1}d_{k+3}}{b_{k+7}}\\
\implies d_{k+3}&=b_{k+4}b_{k+7}.\end{align*}
Thus, for all $n\in\mathbb{N}$ with $n\equiv 0\pmod 3$, our claim is true by induction. The other two cases modulo $3$ are identical.
\end{proof}

This is sufficient to see integrality. (The structure of the following proof was inspired by a proof of the integrality of the Somos-5 sequence given in \cite{Wemyss}.)

\begin{proposition}\label{integer} For $n\geq 3$, both $b_n\in\mathbb{Z}$ and $(b_n,b_{n-3})=(b_n,b_{n-2})=(b_n,b_{n-1})=1.$\end{proposition}
\begin{proof} Proceed by induction. Note that the first $4$ terms are $1,1,2,1\in\mathbb{Z}$, so the base case is true.

Suppose for all $n\leq k+4$ for some $k$ that $b_n$ is integral and the coprime condition is true. Note since $k+1<k+4$ that $(b_{k},b_{k+1})=1$. By Bezout's lemma, $\exists r,s\in\mathbb{Z}$ such that $1=rb_k+sb_{k+1}\implies b_{k+5}=rb_kb_{k+5}+sb_{k+1}b_{k+5}$.

By the definition of the sequence, for some coefficient $c_1\in\{1,3\}$, 
$b_{k+1}b_{k+5}=b_{k+4}b_{k+2}-c_1b_{k+3}^2$. By the corollary above, 
$$b_kb_{k+5}-b_{k+2}b_{k+3}=d_k=c_2b_{k+1}b_{k+4}\implies b_kb_{k+5}=c_2b_{k+1}b_{k+4}+b_{k+2}b_{k+3}$$ for some $c_2\in\{1,3\}$. Therefore, by the inductive hypothesis we have $b_{k+1}b_{k+5},b_kb_{k+5}\in\mathbb{Z}$, hence $b_{k+5}=rb_kb_{k+5}+sb_{k+1}b_{k+5}\in\mathbb{Z}$ as desired.

To see that $b_{k+5}$ is coprime to the three terms before it, note that $(b_{k+1}b_{k+4},b_{k+2}b_{k+3})=1$ because by the inductive hypothesis, both $b_{k+1}$ and $b_{k+4}$ are coprime to $b_{k+2}b_{k+3}$. Thus,\begin{align*}
(b_{k+1}b_{k+4},b_{k+2}b_{k+3})&=1\\
(b_{k+1}b_{k+4},c_2b_{k+1}b_{k+4}+b_{k+2}b_{k+3})&=1\\
(b_{k+1}b_{k+4},b_{k}b_{k+5})&=1.\end{align*}
Therefore, $b_{k+5}$ is coprime to $b_{k+4}$ as desired. 

Similarly, $b_{k+1}b_{k+5}=b_{k+4}b_{k+2}+\{-1,-3\}\cdot b_{k+3}^2$ by the first definition, meaning that $(b_{k+1}b_{k+5},b_{k+3}^2)=(b_{k+4}b_{k+2},b_{k+3}^2)=1$, and so $b_{k+5}$ is coprime to $b_{k+3}$ as well. 

Last, $b_{k-2}b_{k+5}=-b_{k-1}b_{k+4}+5b_{k+1}b_{k+2}$ by the second definition, meaning that $(b_{k-2}b_{k+5},b_{k+1}b_{k+2})=(b_{k-1}b_{k+4},b_{k+1}b_{k+2})=1$, and so $b_{k+5}$ is coprime to $b_{k+2}$ as well.

Therefore, $b_{k+5}$ is coprime to the previous three terms as desired.
\end{proof}

Lastly, knowing that the ECHO sequence is integral, the proof of our main lemma of section~\ref{sequences}, lemma~\ref{twoNplusthree}, follows identically as before.

\begin{lemma}
Define $P = (4,7)$ on E : $y^{2} + y = x^{3} - 3x + 4$. Then for $n\geq 0$,
$$(2n+1)P=\left(\frac{g(n)}{b_n^2},\frac{f(n)}{b_n^3}\right),$$
where $g(n)=2b_n^2-b_{n-3}b_{n+3}$, and
$$f(n)=\begin{cases}
b_n^3+3b_{n-1}^2b_{n+2}&\mathrm{if}~n\equiv 0\pmod 3,\\
b_n^3+b_{n-1}^2b_{n+2}&\mathrm{if}~n\equiv 1\pmod 3,\\
b_n^3+9b_{n-1}^2b_{n+2}&\mathrm{if}~n\equiv 2\pmod 3.
\end{cases}$$\end{lemma}

\bibliographystyle{plain}
\bibliography{echo}

\begin{thebibliography}{10}

\bibitem{MR1484478}
Wieb Bosma, John Cannon, and Catherine Playoust.
\newblock The {M}agma algebra system. {I}. {T}he user language.
\newblock {\em J. Symbolic Comput.}, 24(3-4):235--265, 1997.
\newblock Computational algebra and number theory (London, 1993).

\bibitem{MR3236783}
David~A. Cox.
\newblock {\em Primes of the form {$x^2 + ny^2$}}.
\newblock Pure and Applied Mathematics (Hoboken). John Wiley \& Sons, Inc.,
  Hoboken, NJ, second edition, 2013.
\newblock Fermat, class field theory, and complex multiplication.

\bibitem{Somos5}
Bryant Davis, Rebecca Kotsonis, and Jeremy Rouse.
\newblock The density of primes dividing a term in the {S}omos-5 sequence.
\newblock Preprint.

\bibitem{MR2995149}
Tim Dokchitser and Vladimir Dokchitser.
\newblock Surjectivity of mod {$2^n$} representations of elliptic curves.
\newblock {\em Math. Z.}, 272(3-4):961--964, 2012.

\bibitem{MR2061214}
Henryk Iwaniec and Emmanuel Kowalski.
\newblock {\em Analytic number theory}, volume~53 of {\em American Mathematical
  Society Colloquium Publications}.
\newblock American Mathematical Society, Providence, RI, 2004.

\bibitem{MR2640290}
Rafe Jones and Jeremy Rouse.
\newblock Galois theory of iterated endomorphisms.
\newblock {\em Proc. Lond. Math. Soc. (3)}, 100(3):763--794, 2010.
\newblock Appendix A by Jeffrey D. Achter.

\bibitem{MR789184}
J.~C. Lagarias.
\newblock The set of primes dividing the {L}ucas numbers has density {$2/3$}.
\newblock {\em Pacific J. Math.}, 118(2):449--461, 1985.

\bibitem{MR1251907}
J.~C. Lagarias.
\newblock Errata to: ``{T}he set of primes dividing the {L}ucas numbers has
  density {$2/3$}'' [{P}acific {J}. {M}ath.\ {\bf 118} (1985), no.\ 2,
  449--461; {MR}0789184 (86i:11007)].
\newblock {\em Pacific J. Math.}, 162(2):393--396, 1994.

\bibitem{MR0103854}
W.~Sierpi{\'n}ski.
\newblock Sur une d\'ecomposition des nombres premiers en deux classes.
\newblock {\em Collect. Math.}, 10:81--83, 1958.

\bibitem{MR1312368}
Joseph~H. Silverman.
\newblock {\em Advanced topics in the arithmetic of elliptic curves}, volume
  151 of {\em Graduate Texts in Mathematics}.
\newblock Springer-Verlag, New York, 1994.

\bibitem{MR2514094}
Joseph~H. Silverman.
\newblock {\em The arithmetic of elliptic curves}, volume 106 of {\em Graduate
  Texts in Mathematics}.
\newblock Springer, Dordrecht, second edition, 2009.

\bibitem{sage}
W.\thinspace{}A. Stein et~al.
\newblock {\em {S}age {M}athematics {S}oftware ({V}ersion 6.6)}.
\newblock The Sage Development Team, 2015.
\newblock {\tt http://www.sagemath.org}.

\bibitem{PARI2}
{The PARI~Group}, Bordeaux.
\newblock {\em {PARI/GP version {\tt 2.7.0}}}, 2014.
\newblock available from \url{http://pari.math.u-bordeaux.fr/}.

\bibitem{Wemyss}
Michael Wemyss.
\newblock A very elementary proof that the {S}omos 5 sequence is integer
  valued.
\newblock Unpublished. Available at
  \url{http://www.maths.ed.ac.uk/~wemyss/Somos5proof.pdf}.

\end{thebibliography}

%

\end{document}